\documentclass[a4paper,12pt]{article}

\usepackage{mathptmx}      

\usepackage{latexsym}
\usepackage[english]{babel}
\usepackage{amsmath}
\usepackage{amsfonts}
\usepackage{amssymb}
\usepackage{amsthm}
\usepackage{mathrsfs} 
\usepackage{graphicx}
\usepackage[all]{xy}

\usepackage{caption} 

\usepackage{cite} 

\usepackage{microtype}
\usepackage[super]{nth}

\usepackage[usenames,dvipsnames]{xcolor} 
\usepackage{enumerate}

 \usepackage{hyperref}
 \hypersetup{
     bookmarks=true,         
     bookmarksopen=true,    %
     bookmarksopenlevel=1, %
     pdftoolbar=true,        
     pdfmenubar=true,        
     pdffitwindow=true,     
     pdfstartview={FitH},    
     pdftitle={Elimination of extremal index zeroes from generic paths of closed \texorpdfstring{$1$}{1}-forms},    
     pdfauthor={Carlos Moraga Ferr\'{a}ndiz},     
     pdfkeywords={Morse-Novikov theory, Cerf theory, closed \texorpdfstring{$1$}{1}-forms, pseudo-gradients, \texorpdfstring{$s$}{s}-cobordism}, 
     pdfnewwindow=true,      
    colorlinks=true,       
     linkcolor=Emerald,          
     citecolor=OliveGreen,        
     filecolor=magenta,      
     urlcolor=Bittersweet           
 }

\newtheorem{Teo}{Theorem}
\newtheorem{definition}[Teo]{Definition}
\newtheorem{proposition}[Teo]{Proposition}

\newtheorem{lemma}[Teo]{Lemma}
\newtheorem{nota}[Teo]{Notation}
\newtheorem{example}[Teo]{Example}
\newtheorem{remark}[Teo]{Remark}

\DeclareMathOperator{\Bas}{Bas}

\DeclareMathOperator{\Crit}{Crit}
\DeclareMathOperator{\Diff}{Diff}

\DeclareMathOperator{\Gra}{Gr}

\DeclareMathOperator{\Id}{Id}
\DeclareMathOperator{\Int}{int}

\DeclareMathOperator{\im}{Im} 
\DeclareMathOperator{\ind}{ind}
\DeclareMathOperator{\lat}{lat}

\DeclareMathOperator{\loc}{loc}
\DeclareMathOperator{\mo}{mod}

\DeclareMathOperator{\rk}{rk}

\DeclareMathOperator{\supp}{supp}

\newcommand{\bD}{\mathbb{D}}

\newcommand{\bR}{\mathbb{R}}
\newcommand{\bS}{\mathbb{S}}

\newcommand{\bZ}{\mathbb{Z}}

\newcommand{\cC}{\mathcal{C}}

\newcommand{\cF}{\mathcal{F}}

\newcommand{\cL}{\mathcal{L}}
\newcommand{\cM}{\mathcal{M}}

\newcommand{\cR}{\mathcal{R}}
\newcommand{\cS}{\mathcal{S}}
\newcommand{\cT}{\mathcal{T}}

\newcommand{\al}{\alpha}
\newcommand{\be}{\beta}
\newcommand{\de}{\delta}
\newcommand{\ga}{\gamma}

\newcommand{\lam}{\lambda}

\newcommand{\ve}{\varepsilon}
\newcommand{\vp}{\varphi}
\newcommand{\Lam}{\Lambda}

\newcommand{\wt}{\widetilde}

\newcommand{\rra}{\longrightarrow}

\newcommand{\adh}{\overline}
\newcommand{\co}{\circ}

\newcommand{\gli}{\boldsymbol{\ell}}

\newcommand{\p}{\partial}

\newcommand{\prive}{\smallsetminus}
\newcommand{\trans}{\pitchfork}

\newcommand{\x}{\times}
\newcommand{\V}{\varnothing}

\newcommand{\di}{\displaystyle}

\newcommand{\esp}{\,\,\,}
\newcommand{\miniesp}{\,\,}

\newcommand{\neee}{\negmedspace}

\newcommand{\champs}{(\xi_t)_{\param}}

\newcommand{\formes}{(\al_t)_{\param}}

\newcommand{\param}{ t\in [0,1] }
\newcommand{\primits}{(h_t)_{\param}}

\providecommand{\abs}[1]{\lvert#1\rvert}

\providecommand{\flecur}[1]{\ar@/^#1pc/ }

\providecommand{\flee}[1]{\di \mathop{\rra}^{#1}}

\providecommand{\norm}[1]{\lVert#1\rVert}

\providecommand{\parent}[1]{\left( #1 \right) }

\providecommand{\bparent}[1]{\bigl( #1 \bigr) }

\providecommand{\sing}[1]{\left\lbrace #1\right\rbrace }

\newcommand{\ch}[2]{ (#1)_{#2 \in [0,1]} } 

\newcommand{\ens}[2]{ \left\lbrace  #1 \ \middle| \ #2 \right\rbrace }

\newcommand{\scal}[2]{ \langle  #1, #2 \rangle } 

\newcommand{\quotient}[2]{\raise1ex\hbox{$#1$}\Big/ \lower1ex\hbox{$#2$}}

\newcommand{\ensdefTres}[3]{ #1 :=  \left\lbrace  #2 \ \middle| \ #3 \right\rbrace } 

\newcommand{\chext}[4]{ (#1)_{#2 \in [#3,#4]} }

\newcommand{\ensCuatro}[4]{ #1 =  \left\lbrace  #2 \ \middle| \  \begin{array}{l}
#3 \\
#4 
\end{array} \right\rbrace } 

\newcommand{\lift}[4]{\[ 
   \xymatrix{
  & \wt{#2} \ar[d]^{#4}  \\
#1 \ar[r]^#3 \ar@{.>}[ur]^{\wt{#3}} & #2
} 
\]}

\newcommand{\aplic}[5]{
\begin{array}{rccc}
#1 : & #2 & \rra & #3 \\
  &  #4 & \longmapsto & #5
\end{array} }

\newcommand{\applic}[5]{
$$\begin{array}{rccc}
#1 : & #2 & \rra & #3 \\
  &  #4 & \longmapsto & #5
\end{array}$$ }

\newcommand{\cuadr}[8]{
 \xymatrix{
 #1 \ar[r]^{#5}\ar[d]_{#6} & #2 \ar[d]^{#7} \\
 #3 \ar[r]^{#8} & #4
}
 }

  \newcommand{\sistIVU}[8]{ \left\{
  \begin{array}{ll}
       #1 &  #2 \\
        #3 &  #4 \\
  #5 &  #6 \\
 #7 &  #8
  \end{array}
  \right. }

\newcommand{\univcolim}[9]{
\[ \xymatrix{
 #1 \ar[rr]^{#4}_{#7}\ar[rd]^{#5}_{#8} & & #2 \\
  & #3 \ar@{.>}[ur]^{#6}_{#9} &
}
\] }

\newcommand{\DibLocScalNomEtiq}[5]{
\begin{figure}[#2]
\centering
\includegraphics[scale=#3]{#1}
\caption{#4}\label{#5}
\end{figure}
}

\newcommand{\asc}[1]{
 W^s(#1)
  }

\newcommand{\desc}[1]{
 W^u(#1)
  }

\newcommand{\binap}[1]{
 W^{s/u}(#1)
  }

\newcommand{\ascT}[2]{
 W^s_{#2}(#1)
  } 

\newcommand{\descT}[2]{
 W^u_{#2}(#1)
  }

\begin{document}
\title{Elimination of extremal index zeroes from generic paths of closed \texorpdfstring{$1$}{1}-forms}

\author{C. Moraga Ferr\'{a}ndiz \thanks{e-mail:  carlos@ms.u-tokyo.ac.jp\newline
 }}

\maketitle
\begin{abstract}
Let $ \al $ be a Morse closed $ 1 $-form of a smooth $ n $-dimensional manifold $ M $. The zeroes of $ \al $ of index $ 0 $ or $ n $ are called \emph{centers}. It is known that every non-vanishing de Rham cohomology class $ u $ contains a Morse representative without centers. The result of this paper is the one-parameter analogue of the last statement: every generic path $ \formes $ of closed $ 1 $-forms in a fixed class $ u\neq 0 $ such that $ \al_0,\al_1 $ have no centers, can be modified relatively to its extremities to another such path $ \ch{\be_t}{t} $ having no center at all.
\end{abstract}

\section{Introduction and main result}
Let $M$ be a closed smooth manifold of dimension $n$ and $u$ be a non-zero de Rham cohomology class of degree $1$ of $M$. We are considering $\formes$ a path of closed $1$-forms where $[\al_t]=u$ for every $ t $. Generically, such a path only consists of Morse $1$-forms but in a finite set of times $\sing{t_i}_{i=1}^s$ where the path crosses transversely the co-dimension one strata of birth/elimination closed $1$-forms. Namely, such an $\al_{t_i}$ presents a unique degenerate zero $p\in M$ such that for some coordinate chart $U$ of $p$ and every $t$ near to $t_i$ we have: \[ \al_{t}|_U=d\parent{x_1^3\mp(t-t_i)x_1+Q(x_2,\ldots,x_{n})}. \] Here $Q$ is a non-degenerate quadratic form in the last $n-1$ variables. The birth/elimination strata are naturally co-oriented: we will say that $ \al_{t_i} $ is of birth-type if the orientation induced by the path represents the fixed co-orientation of the strata; it corresponds to the ``minus'' case in the last formula, where a pair of Morse critical zeroes of consecutive index appears as $ t $ increases over $ t_i $.\\
The genericness of  these properties for a path is a direct consequence of Thom's transversality theorem in jet spaces (\neee\cite{thomSing} or \cite{gui+gol} for a more educational presentation), as it was shown by Cerf in \cite{cerfSphere} for the case $u=0$ (see also \cite{laudenbachMorseFloer}).\\

Among the zeroes of a Morse closed $1$-form $\al$, that we denote by $ Z(\al) $, those of extremal index are called \emph{centers}. 
The idea of cancelling pairs of critical points of consecutive index which are connected by a unique (pseudo)gradient trajectory, so much exploited in proving Smale's theorem of $h$-cobordism (see \cite{milnorHcob}), is used in the paper \cite{arnoux+levitt} to construct $\al'$, a Morse closed $1$-form cohomologous to $\al$ which does not have any center. Their only restriction is to ask the cohomology class $u$ not to be rational ($\rk(u)>1$). Latour gives a more topological proof without any requirement on $u\neq 0$ in \cite{latour}. We deal with the one-parameter version of this result and obtain the following theorem:

\begin{Teo}\label{th:ElimParam}
Let $ \dim(M)\geq 3 $. Every generic path of cohomologous closed $ 1 $-forms $ \formes $  such that $ \al_0,\al_1 $ have no centers can be deformed into a path $ \ch{\be_t}{t} $ of the same kind which has the same extremities and no center at all.
\end{Teo}

A natural way to distinguish isotopy classes of non-singular closed $1$-forms in the same cohomology class, is to study deformations of paths of cohomologous  $1$-forms relative to their extremities, the latter being non-singular. One expects to obtain a $K$-theoretic algebraic obstruction to isotopy on the non-exact case, in the spirit of \cite{hat+wag}. This problem is partially treated on \cite{yoTesis}, where we 
provide a process which would help to define such an obstruction for generic paths only containing two intermediate consecutive indexes. The mentioned process, that we call \lq\lq graft of swallow-tail loops", requires theorem \ref{th:ElimParam} of this paper to work. Going one step further -- to paths whose critical indexes are contained on the region $\sing{2,\ldots,n-2}$ -- together with the graft of swallow-tail loops would give a way to define the obstruction.\\
Since lemmas \ref{lem:lips} and \ref{lem:SwallowTail} easily generalize to arbitrary index, they should be useful to shrinking further the range of critical indexes in the context of Novikov acyclicity.\\

\section{Novikov homology}\label{sect:MNHomology}
Take $ \al $ a Morse closed $ 1 $-form and choose $ B_* \subset\wt{M}$ a lifting of $ Z(\al) $ to the universal cover $ \wt{M}\flee{\pi}M $, where $ P\in B_* $ corresponds bijectively to $ p\in Z(\al) $. We remark that $ B_* $ is graduated by the index. From that, we derive a graded module $ C_*(\al) $ freely generated by $ B_* $ over the Novikov ring $ \Lam_{-u} $ associated with the cohomology class of $ \al $. We recall that $ \Lam_{-u} $ is the completion of the group ring $ \Lam:=\bZ[\pi_1M] $ given by \[ \ensCuatro{\Lam_{-u}}{\sum_{g\in\pi_1M}n_gg,n_g\in\bZ}{n_g=0\text{ except for finitely many }g}{\text{or }\esp\lim\limits_{n_g\neq 0}u(g)=-\infty}. \]
As for Morse functions, a differential for $ C_*(\al) $ can be obtained if we give us $ \xi $, a special kind of vector field adapted to $ \al $. 

\begin{definition}\label{df:Pseudo} 
A vector field $ \xi $ is a \emph{pseudo-gradient} of a Morse closed $ 1 $-form $ \al $ if the function $ \al(\xi)\in\cC^{\infty}(M) $ 
\begin{enumerate}[(1)]
\item is strictly negative outside $ Z(\al) $ and\label{cd:Lyap}
\item $ Z(\al) $ are non-degenerate maxima of $ \al(\xi) $.
\end{enumerate} 
\end{definition}

Such a vector field vanishes only at $ Z(\al) $. Each zero $ p\in Z(\al) $ has a local stable and unstable manifold determined by $\xi$; they are denoted respectively by $ \ascT{p}{\loc},\descT{p}{\loc} $ and are diffeomorphic to Euclidean spaces of complementary dimension. Their intersection is transverse and reduced to $ p $. The global stable and unstable manifold of $ p\in Z(\al) $ determined by $\xi$ are defined by 
\[W^s(p):=\ens{x\in M}{\lim_{t\to +\infty}\xi^t(x)=p}\miniesp ; \miniesp W^u(p):=\ens{x\in M}{\lim_{t\to -\infty}\xi^t(x)=p},
\] where $(\xi^t)_{t\in\bR}$ denotes the flow of $\xi$. Given $P$ a lifting of $p$, we denote $ \binap{P} $ the connected component of $ \pi^{-1}(\binap{p}) $ containing $ P$.\\

We denote the set of orbits of $\xi$ going from $p$ to $q$ by $\cL(p,q)$. The image of these orbits is clearly contained in $ \desc{p}\cap\asc{q} $. The choice $B_*$ allows one to define the \emph{enwrapment}, which is a map $ \cL(p,q)\flee{e}\pi_1M $. Remark that every $\ell\in\cL(p,q)$ has a unique lifting $\wt{\ell} $ starting from $P$; $\wt{\ell}$ goes so to $g_{\ell}Q$ for a unique $ g_{\ell}\in\pi_1M $. We set $ e(\ell):=g_{\ell} $.

\begin{definition}
A pseudo-gradient $ \xi $ of a Morse closed $ 1 $-form $ \al $ is \emph{Morse-Smale} if its global stable and unstable manifolds intersect transversely:
\[ \desc{p}\trans\asc{q},\esp\text{ for all }\esp p,q\in Z(\al). \]
\end{definition}

This class of vector fields allows one to count orbits from $ p $ to $ q $, two zeroes of $ \al $ such that $ \ind(p)=ind(q)+1 $. The orbits  $ \ell\in\cL(p,q) $ depend on the choice of such a $ \xi $ and their enwrapment $ e(\ell) $ on the choice of the lifting $B_*$. A sign $ s(\ell)=\pm 1 $ can be computed if an additional choice of orientation of the unstable manifolds is made.  We obtain the incidences $\scal{P}{Q}^{\xi,B}:=\sum_{\ell\in\cL(p,q)}s(\ell)e(\ell) \in\Lam_{-u}$. The property of being Morse-Smale for a pseudo-gradient is generic as it was proven in \cite{kupka}; the interested reader is sent to \cite{pajitnovCircle}. The next theorem is due to Latour:

\begin{Teo}[\neee\cite{latour}]\label{th:Differential}
The $ \Lam_{-u} $-linear extension of the map 
\applic{\p^{\xi,B}_{*+1}}{C_{*+1}(\al)}{C_{*}(\al)}{P}{\di\sum_{Q\in  B_{*}(\al)}\scal{P}{Q}^{\xi,B}Q} is a differential for the graded $ \Lam_{-u} $-module $ C_*(\al) $. Moreover, the complex $ C_*(\al,\p^{\xi,B}_*) $ does not depend on the choice of the triple $(\al,\xi,B_*)$ up to \emph{simple-equivalence}.
\end{Teo}

The notion of simple-equivalence can be found on \cite{maumaryTypesimple}. The homology of the complex $ C_*(\al,\p^{\xi,B}_*) $, which we denote by $ H_*(M,u) $, only depends on the class $u$ and it is called the \emph{Novikov homology} of the class $u$. The historical reason is that the first theorem with the flavour of theorem \ref{th:Differential} was given by Novikov in his foundational paper \cite{novikovMultivalued}, which first gave Morse-type inequalities for $\bS^1$-valued functions $f:M\to\bS^1$. These inequalities are related with a homology theory\footnote{In the mentioned paper, Novikov called \lq\lq semi-open" his homology.} of an \emph{abelian} cover associated to $f$. Later, Sikorav proved in \cite[Ch. IV]{sikoravTesis} that the homology defined by Novikov is indeed a homology with local coefficients and extended it to \emph{non-abelian} covers.\\
Latour also proved that $ H_*(M,u) $ coincides with the version of Novikov homology on the universal cover defined in Sikorav's thesis.\\
Further versions of theorem \ref{th:Differential} can be found on \cite[Ch. 14, Th. 2.2 and Th. 2.4]{pajitnovCircle} and on \cite[Th. 3.1]{farberTopoClosed}.

\section{Connecting saddles and elimination of centers}\label{sect:ConSaddElim0}

We are going to consider \emph{sinks} (centers of index $ 0 $). The case of \emph{sources} (centers of index $ n $) can be treated in an analogous way.\\
A Morse closed $ 1 $-form $ \al $ induces a singular foliation $ \cF_{\al} $ in $ M $: two points $ x,y $ belong to the same leaf $ F $ if there exists a smooth path $ [0,1]\flee{\ga}M $ joining $ x $ to $ y $ such that $ \al(\ga'(t))=0 $ for all $ t\in [0,1] $. A leaf is called singular if it contains some zero of $ \al $.\\
Consider now $ \xi $ a pseudo-gradient for $ \al $. The orbits $ \ell\in\cL^{\xi}(p,q) $ of $ \xi $ going from $ p $ to $ q $, two zeroes of $\al $, have an associated \emph{transverse length} $ L(\ell) $. This length is given by the integral \[ L(\ell):=-\int\ell^*(\al), \] which is positive thanks to the condition \emph{(\ref{cd:Lyap})} of the definition \ref{df:Pseudo} of a pseudo-gradient.\\

Remark that $ u $ induces a group morphism $ \pi_1M\flee{u}\bR $ given by integrating $ \al $ along a representative $ \gamma $ of a loop $ g=[\gamma] $. As $\pi^*\al$ is exact, let us take a \emph{primitive} $\wt{M}\flee{h}\bR$, which verifies $dh=\pi^*\al$. An easy calculation shows that for $\ell\in\cL(p,q)$ we have: 
\begin{equation}\label{eq:HeightEnwrap}
L(\ell)=L(\wt{\ell})=h(P)-h(e(g_{\ell})Q)=h(P)-h(Q)-u(e(g_{\ell})),
\end{equation}
where $\wt{\ell}$ denotes the lifted orbit starting from $P$. This equality relates numerically the choice of liftings with the enwrapment of orbits.

\begin{nota}\label{nota:SaddSeparat}
{\rm We call \emph{saddles} $ \cS:=Z_1(\al) $ the index $ 1 $ zeroes of $ \al $. The unstable manifold of a saddle $s$ associated with a pseudo-gradient $\xi$ always decomposes as $ \desc{s}=\sing{s}\cup\ell^+\cup\ell^-$  where $\ell^+,\ell^-$ are two different non-trivial orbits called \emph{separatrices}. }
\end{nota}

A well known fact about Novikov homology is that $H_0(M,u)=0$ for every $u\neq 0$ (see \cite[Cor. 1.33]{farberTopoClosed} for example). So, for $ \xi $ Morse-Smale, it implies in particular that the set $ \cS_c $ of saddles with at least one separatrix going to $ c $ is non-empty: otherwise the lifting $C$ has no chance to be in $\im(\p_1^{\xi})$ and we would have $0\neq[C]\in H_0(M,u) $.\\
We distinguish two kinds of saddles $ s\in\cS_c $.\\
~\newline
\begin{minipage}{0.7\textwidth}
The saddle $ s\in\cS_c $ is said to be of type  \emph{ker} if $ \cL(s,c)=\sing{\ell^+,\ell^-} $ and $ L(\ell^+)= L(\ell^-) $. Any lifting of $ \desc{s} $ determines two liftings of $ c $ related by an element $ g\in\ker(u) $ as in figure \ref{fig:TypeKer}: this is a direct consequence of relation \eqref{eq:HeightEnwrap}. These collection of saddles is denoted by $ \cS_c^{\ker} $. Saddles in the complementary set $\cS_c^n:= \cS_c\prive \cS_c^{\ker}$ are called \emph{normal saddles}. Remark that if $ C\in B_* $ is the chosen lifting of a sink $ c $, we can find a lifting $ S $ of $ s\in\cS_c $ such that $ \desc{S} $ is as in figure \ref{fig:TypeKer} or \ref{fig:Type1}. However, there is a priori no reason to have $ S\in B_* $.
\end{minipage}
\hspace{5pt}
\vrule
\begin{minipage}{0.25\textwidth}
\centering
\includegraphics{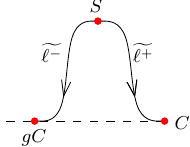}
\captionof{figure}{$s\in\cS_c^{\ker} $}
\label{fig:TypeKer}
\end{minipage}
\DibLocScalNomEtiq{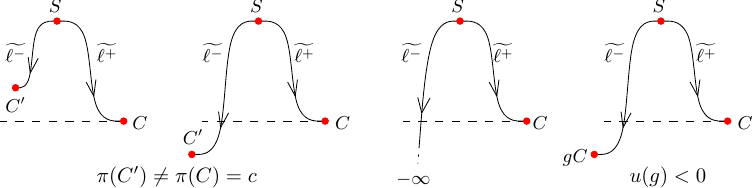}{h!}{.9}{Possible situations of a lifting of $ \desc{s} $ to $ \wt{M} $ for $ s\in\cS_c^n $.}{fig:Type1}

If a normal saddle $ s\in\cS_c^n $ has its two orbits going to $ c $, their transverse lengths\footnote{Since there is no other notion of length associated with orbits in the present paper, we will in the sequel forget the adjective ``transverse'' relative to this length.} are necessarily different; for any normal saddle $s$, we note by $ \ell^+ $ the shortest orbit joining $s$ to $ c $ and call it \emph{the connecting orbit} of $s$. We note by $\cL^+_c$ the finite set of connecting orbits from the set of normal saddles $s\in\cS_c^n$ to $c$. 

\begin{definition}\label{df:ConnectSadd}
A \emph{connecting saddle} for the sink $ c $ is a normal saddle $s\in\cS_c^n  $ minimizing the length of connecting orbits $ \cL_c^+\flee{L}\bR^+ $. In other words: $ L(\ell^+_s)\leq L(\ell^+_{s'}) $ for every $ s'\in\cS_c^n $.
\end{definition}

Connecting saddles are going to be the main tool to eliminate sinks. We justify their existence by a purely algebraic argument in lemma \ref{lem:ConnectSadd} below.

\begin{lemma}\label{lem:ConnectSadd}
Let $ \al $ be a Morse closed non-exact $ 1 $-form together with a Morse-Smale pseudo-gradient $ \xi $. Then, every sink $c\in Z_0(\al)$ admits a connecting saddle.
\end{lemma}
\begin{proof}
Choose liftings of $c$ and of saddles $\cS$; for saddles in $\cS_c$ we do that as  in figures \ref{fig:TypeKer} and \ref{fig:Type1}. Let us see that $ \cS_c=\cS_c^{\ker}$ would imply that $ C\notin\im(\p^{\xi}_1) $, leading to the same contradiction as before. By identifying saddles with their chosen liftings, we clearly have $ C\notin\p_1^{\xi}\parent{\Lam_{-u}\otimes(\cS\prive{\cS_c})} $. An element of $ \im\parent{\p_1^{\xi}|_{\Lam_{-u}\otimes(\cS_c^{\ker})}} $ is written as a finite sum times $ C $: $$\mu\cdot C:=\sum_{s_i\in\cS_c^{\ker}}\lam_i(\pm 1\pm g_i)\cdot C, $$ where $ g_i\in\ker(u) $ and $ \lam_i\in\Lam_{-u} $. The next map is a morphism of $\bZ[\ker(u)]$-modules: \[\aplic{(\cdot)^*}{\Lam_{-u}}{\bZ[\ker(u)]}{\lam}{\sum\limits_{g\in\ker(u)}n_gg}\esp . \] Since $ (\pm 1\pm g_i)=(\pm 1\pm g_i)^* $ for all $ i $, we have $ \mu^*=\sum\lam_i^*(\pm 1\pm g_i) $.  Consider $ I $ the kernel of the augmentation morphism of rings $ \bZ[\ker(u)]\flee{\ve_2}\bZ_2 $ given by $ ng\mapsto n\mo 2$. As the terms $ 
 (\pm 1\pm g_i) $ belong to the augmentation ideal, also does $ \mu^* $. Then $ \mu $ cannot be $ 1 $. 
\end{proof}

\begin{definition}\label{df:Excellent}
A Morse closed $ 1 $-form $ \al $ is said to be $0$-\emph{excellent} if there exists a pseudo-gradient $\xi$ for $\al$ such that for every sink $ c\in Z_0(\al) $, the map $\cL_c^+\flee{L}\bR^+ $ is injective.
\end{definition}

\begin{remark}\label{rem:0ExcelOpenGener}
The property of $0$-excellence guarantees the \emph{uniqueness} of connecting saddles: the connecting saddle of a sink $c$ is the only $s\in\cS_c^n$ such that $\ell^+_s$ realizes the minimum of  $\cL_c^+\flee{L}\bR^+ $.  Moreover, $0$-excellence is generic for $\al$ because any Morse closed $ 1 $-form can be slightly perturbed in order to obtain a cohomologous $0$-excellent $ 1 $-form by application of lemma \ref{lem:Arrang} to saddles.
\end{remark}

\begin{lemma}[{\sc Rearrangement lemma}]\label{lem:Arrang}
Let $p$ be a zero of a Morse closed $1$-form $\al_0$ of index $ k $, equipped with a primitive $ \wt{M}\flee{h_0}\bR $ and a pseudo-gradient $ \xi $; consider $P$ a lifting of $p$ to $\wt{M}$. If $K>K'>0$ are such that \[ \desc{P}\cap h_0^{-1}\parent{[h_0(P)-K,+\infty)} \text{ is an embedded disk } \bD^k,\] then there exists $U$ a neighbourhood of $p$ and a path $\formes$ of cohomologous Morse closed $1$-forms with primitives $\primits$ such that for all $t\in [0,1]$ we have:
\begin{enumerate}
\item $ \xi $ is pseudo-gradient of $ \al_t $,
\item $U\cap Z(\al_t)={p}$, 
\item $ \al_t=\al_0 $ on $ M\prive U$ and
\item $ h_1(P)=h_0(P)-K'$.  
\end{enumerate}  
\end{lemma}
\begin{proof}
In the context of real-valued functions, this lemma is classical and can be found in \cite[4.1]{milnorHcob} for example. For closed non-exact  $1$ forms, it is also well known and can be found in \cite[9.5.1]{farberTopoClosed} under a slightly different presentation.  The hypothesis about $\desc{P}$ allows one, if $k>1$, to take the neighbourhood $U$ as a thickening of $\ascT{p}{\loc}\cup\bD^k $ in such a way that it has lateral boundary  $\p_{\lat}U$ diffeomorphic to $\bS^{k-1}\x\bS^{n-k-1}\x [0,1]$ where each $(x,y)\x [0,1]$ is contained in an orbit of $\xi$. A properly chosen isotopy of $[0,1]^2$ decreases the values of the primitive on a collar of $\p_{\lat}U$ inside $U$ as claimed; this is standard and carefully proven in \cite[Lemme 2.2.34]{yoTesis}. Since we deal principally with sinks ($k=0$), and the mentioned choice of  $U$ is no more possible in this case, we present  here the proof for sinks. Remark that sinks automatically verify the hypothesis about $\desc{P}=\sing{P}$ for every $ K>0 $: there is no obstruction to decrease the values of a primitive as much as we want near local minima.\\

Choose $\xi$ a Morse-Smale pseudo-gradient for $\al_0$ and call $m:=h_0(C) $. Define $U:=\ascT{c}{\loc}$ and take $V$ the connected component of $\pi^{-1}(U)$ containing $C$. We provide  $\bD^n_{\ve}$, the closed $n$-disk of radius $\ve>0$, with polar coordinates $ (\theta,r)\in\frac{\bS^{n-1}\x[m,m+\ve]}{\bS^{n-1}\x\sing{m}}$. By taking an $\ve>0$ small enough, we have a diffeomorphism $\Psi:\bD^n_{\ve}\to V\cap h_0^{-1}\parent{(-\infty, m+\ve]} $ such that:
\begin{enumerate}
\item $h_0\parent{\Psi(\bS^{n-1}\x\sing{r})}=r$, for every $r\in [m, m+\ve]$ and
\item $\Psi\parent{\sing{\theta}\x(m,m+\ve]}$ is a lifting of a trajectory of $\xi$, for all $\theta\in\bS^{n-1}$.
\end{enumerate}

\begin{minipage}{0.4\textwidth}
Take $\ch{\vp}{t}$ an isotopy of $\bR$ such that for all $ t\in[0,1]$:
\begin{enumerate}
\item $\vp_t|_{[m+\frac{\ve}{2},\infty)}=\Id$, 
\item $\vp_t|_{(-\infty, m+\frac{\ve}{4}]}=\Id-tK$.
\end{enumerate}
\end{minipage}
\hspace{5pt}
\vrule
\hspace{5pt}
\begin{minipage}{0.4\textwidth}
\centering
\includegraphics[scale=0.8]{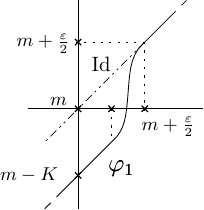}
\captionof{figure}{Graph of $\vp_0$ and $ \vp_1$.}
\label{fig:Graph}
\end{minipage}\\

The path $h_t:=\vp_t\co h_0|_{V}$ extends $ \pi_1M$-equivariantly  to $ V':=\bigcup_{g\in\pi_1M}gV$ and then to the whole $\wt{M}$ by taking $h_t=h_0$ on $\wt{M}\prive V'$. The induced path of $1$-forms $\al_t:=\pi_*(dh_t)$ fulfills the required conditions.\\
The vector field $ \xi $ clearly keeps the property of being a pseudo-gradient all along the path $ \formes $ since the values of $ \al_t(\xi) $ are those of $  \al_0(\xi) $ eventually multiplied by $ \vp_t'>0 $. 
\end{proof}

We only need the cases $ k=0,1 $ of lemma \ref{lem:Arrang} in this paper. We want to eliminate pairs of Morse zeroes of these indexes; the next version of the Morse elimination lemma, which corresponds to the only theorem of \cite{laudenbachElim}, is well adapted to our purpose.  

\begin{lemma}[{\sc Morse elimination lemma}]\label{lem:ElimStd}
Let $N\flee{h_0}\bR$ be a Morse function equipped with a Morse-Smale pseudo-gradient $\xi$. Let $p,q$ two critical points of $h_0$ such that $\ind(p)=\ind(q)+1$. Suppose that there exists $\ve>0 $ such that every orbit of $\desc{p}$ reaches the level $h_0^{-1}\parent{h_0(q)-\ve}$ but one which goes to $q$,\\
then there exists $V$ a neighbourhood of $\adh{\desc{p}}\cap h_0^{-1}\bparent{[h_0(q)-\ve,+\infty)}$ and a generic path of functions $\primits$ such that:
\begin{enumerate}
\item $ h_t $ is Morse for all $ t\neq \frac{1}{2} $, 
\item $ h_t=h_0 $ on $ N\prive V$ for all $t$ and
\item $ \Crit (h_t|_V)=\V$ for all $t>\frac{1}{2}$.
\end{enumerate}
\end{lemma}

\begin{lemma}\label{lem:Elim0}
Let $s$ be a connecting saddle for a sink $c$ of a Morse closed $1$-form $\al_0$. There exists a generic path $\formes$ of cohomologous $ 1 $-forms, beginning at $\al_0$, crossing the birth/elimination strata only once and such that $ Z(\al_1)=Z(\al_0)\prive\sing{s,c} $.
\end{lemma}
\begin{proof}
Equip $\al_0$ with a Morse-Smale pseudo-gradient such that the hypothesis in the present lemma is verified. Consider $\wt{M}\flee{h_0}\bR$ a primitive of $\al_0$; since $s$ is a normal saddle we can take $S,C$, liftings of $s,c$ as in figure \ref{fig:Type1}. We want to apply lemma \ref{lem:ElimStd} relatively to $p=S,q=C$, which are joined by the lifted connecting orbit $\wt{\ell^+}$.
Since $\desc{S}$ has only another orbit, namely $\wt{\ell^-}$ and thanks to equation \eqref{eq:HeightEnwrap}, the hypothesis of lemma \ref{lem:ElimStd} relative to  $\desc{S}$ is equivalent to $L(\ell^-)>L(\ell^+) $. The only situation which does not verify this inequality is that of first picture on figure \ref{fig:Type1}; but in this case $\abs{\cL(s,c)}=1$ and the orbit $\ell^-$ of $ s $ does not go to $c $. We can then apply lemma \ref{lem:Arrang} to $ P=C' $, the ending point of $ \wt{\ell^-} $, in order to decrease the value $h_0(C')$  as much as we want without affecting the value of $h_0$ on $C$; after this modification the situation is that of the second picture of figure \ref{fig:Type1}. In any case, the lifted orbit $\wt{\ell^-}$ goes under the level containing $C$, and the unstable manifold $\desc{S}$ verifies the hypothesis of lemma \ref{lem:ElimStd}. We find $N$ a neighbourhood of $\adh{\desc{S}}\cap h_0^{-1}\bparent{[h_0(C)-\ve,\infty)}$ for $\ve>0$ small enough so that $\pi|_N$ is injective. Apply now the elimination lemma \ref{lem:ElimStd} to $h_0|_{N}$ and the critical pair $S,C$; we perform this deformation of $h_0$ in a $ \pi_1M$-equivariant manner by imposing $h_t|_{gN}=h_t|_{N}+u(g)$ for every $g\in\pi_1M$ and $t\in [0,1]$. So the push-forward  $\al_t:=\pi_*(dh_t)$ for $ t\in [0,1]$, is a well defined path of $1$-forms which has the required properties. 
\end{proof}

Note that lemmas \ref{lem:Arrang}, \ref{lem:ElimStd} and \ref{lem:Elim0} also hold for smooth families of data when the assumptions are satisfied for every parameter. 
The next two lemmas \ref{lem:lips} and \ref{lem:SwallowTail}, are nothing but the adapted versions of \cite[Lemma 3.6 and Lemma 3.5]{laudenbachReid} for non-exact closed $ 1 $-forms, where we restrict to the only case which we are concerned with: zeroes of indexes $0$ and $ 1$. In the mentioned paper, the author proves the \emph{swallow tail lemma}, stating that the proof of the \emph{lips lemma} can be performed in the same way. We proceed in the complementary way. Lemma \ref{lem:SwallowTail} can be proven by adapting the proof of \cite[Lemma 3.5]{laudenbachReid}.

\begin{lemma}[{\sc Elementary lips lemma for sinks}]\label{lem:lips}
Let $ (\al_t, \xi_t)_{\param} $ be generic. Suppose that $ \chext{s_t,c_t}{t}{t_0}{t_1} $ are continuous paths of saddles and sinks such that $ s_{t_0}=c_{t_0},s_{t_1}=c_{t_1} $ are respectively birth and elimination and that there exists an $ \ve>0 $ and a continuous family $\parent{\ell_t^+}_{t\in (t_0,t_1)},\ell_t^+\in\cL(s_t,c_t)$ such that for every $ t\in (t_0,t_1) $ we have: 
\begin{equation}\label{eq:LengthIneq}
L(\ell_t^-)>L(\ell_t^+)+\ve,\text{ where } \ell_t^- \text{ denotes the other separatrix of } \desc{s_t}.
\end{equation}
Then, for all $ \de>0 $ there exists a generic $ \ch{\al_t'}{t} $ such that:
\begin{enumerate}
\item $ \al_t=\al_t' $ for all $ t\notin (t_0-\de,t_1+\de) $,
\item $ Z(\al_t')=Z(\al_t)\prive\sing{s_t,c_t} $ for all $ t\in (t_0-\de,t_1+\de) $.
\end{enumerate}
\end{lemma}
\begin{proof}
We are inspired by the proof of the only theorem of \cite{laudenbachElim}. Choose primitives $ \primits $ and consider a continuous lifting $ \chext{\wt{\ell_t^+}}{t}{t_0}{t_1} $ of the distinguished orbits, which join $ S_t $ with $ C_t $, liftings of the considered zeroes. Thanks to the assumption about $\ell_t^-$, the curve $\desc{S_t}$ intersects the level of $ h_t(C_t)-\ve $ in a unique point $ a_t $ for $t\in [t_0,t_1]$. For each of these times, choose an arc $\lam_t^+:[-\sqrt{\ve},\sqrt{\ve}]\to\asc{C_t}$  verifying $ h_t(\lam_t^+(u))=h_t(C_t)+u^2 $ for every $ u\in[-\sqrt{\ve},\sqrt{\ve}] $ and extending smoothly $\desc{S_t}$ at $ \sing{\lam_t^+(-\sqrt{\ve})}=\wt{\ell_t^+}\cap h_t^{-1}\parent{h_t(C_t)+\ve}$. Denote by $b_t:=\lam_t^+(\sqrt{\ve})$ and by $\descT{S_t}{+}$ the connected and compact arc inside $\desc{S_t}$ determined by the points $a_t$ and $ \lam_t^+(-\sqrt{\ve}) $. \\
Let $a_t:=a_{t_0},a_{t_1}, b_t:=b_{t_0},b_{t_1}$ respectively for $t\in [t_0-\de,t_0),(t_1,t_1+\de]$. It is easy to find a continuous family of smooth arcs $ \chext{[0,1]\flee{I_t}\wt{M}}{t}{t_0-\de}{t_1+\de} $ such that for all $ t\in [t_0-\de,t_1+\de] $:
\begin{enumerate}
\item it parametrizes $ \descT{S_t}{+}\cup \im\parent{\lam^+_t} $ when defined,
\item the only critical points of $ h_t|_{\im(I_t)} $ are $ S_t,C_t $ and have same index that $s_t,c_t$,
\item the extreme values are $ I_t(0)=a_t $ and $ I_t(1)=b_t$.
\end{enumerate}
We remark that for $ t\in [t_0-\de,t_0)\cup (t_1,t_1+\de] $ we can define $ I_t $ by the rescaled flow line of $ -\wt{\xi_t} $  of transverse length equal to $ 2\ve $ starting at $ a_t $; namely $ I_t(u):=\xi_{t}^{-2\ve u}(a_t) $.\\

{\sc Claim:} \textit{There exists a smooth family  $ \chext{N_t}{t}{t_0-\de}{t_1+\de} $ of tubular neighbourhoods of the arcs $ I_t $ in $ \wt{M} $ where $ \pi $ is injective, together with coordinates $ (u,z)\in [0,1]\x\bR^{n-1} $ such that $ I_t(u)=(u,0) $ and
\[ h_t(u,z)=f_t(u)+Q(z).  \]
Here $ Q $ is a positive definite quadratic form and $  f_t:=h_t\co I_t $.}\\

Once the claim is proved, we can conclude as follows:\\
~\newline
\begin{minipage}{0.3\textwidth}
For all $ t\in [t_0-\de,t_1+\de] $, choose a smooth function $ g_t:[0,1]\to\bR $ such that $ g_t\leq f_t $, coincides with $ f_t $ near $ \sing{0,1} $ and haves no critical points as figure \ref{fig:LipsDescent} suggests.
\end{minipage}
\hspace{3pt}
\vrule
\hspace{1pt}
\begin{minipage}{0.65\textwidth}
\centering
\includegraphics[scale=1.1]{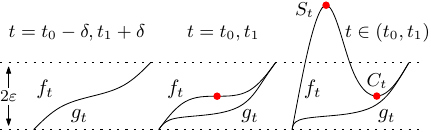}
\captionof{figure}{Graphs of $f_t, g_t$ at the indicated times.}
\label{fig:LipsDescent}
\end{minipage}\\

We want to modify the primitives $h_t=f_t+Q$ into $g_t+Q$, the latter having no critical point on $N_t$ thanks to the choice of $g_t$.
The interpolation, for $ s $ running into $ [0,1] $, given by $ k^s_t:= f_t+s\parent{g_t-f_t}+Q $ almost succeeds: we do not want to change the values of the primitive $h_t$ near $\p N_t$. Pick  $\nu>0$ as small as desired and construct a bump function $ \omega:[0,\infty) \to\bR $ verifying: 
\begin{enumerate}
\item $\omega' \leq 0$,
\item  $\supp (\omega)=[0,2\nu]$ and $\omega |_{[0,\nu]}=1$. 
\end{enumerate}
The two parameter family of functions \[ h_t^s(u,z):=f_t(u)+\omega(\norm{z})\cdot s\bparent{g_t(u)-f_t(u)}+Q(z),  \] clearly coincides with $h_t$ near $\p N_t$ and with $k_t^s$ when $\norm{z}\leq \nu$, in particular near $\im(I_t)=\sing{z=0}$. Moreover, it has the same critical points as $ k^s_t $:  for $z$ such that $\norm{z}\geq \nu$, a calculation shows that for every $i=1,\ldots, n-1$: 
\[
\dfrac{\p h_t^s}{\p z_i}(u,z):=\parent{\dfrac{\omega'(\norm{z})}{\norm{z}} s\bparent{g_t(u)-f_t(u)}+q_i}z_i,
\]
where the term $q_i>0$ comes from the positive definite quadratic form $Q$. The factor coming with $z_i$ is strictly positive due to the assumption on $\omega'$ and to the fact that $g_t-f_t\leq 0$. The  $ z$-derivative of $h_t^s$ vanishes only at $ z=0 $, where $h_t^s$ and $k_t^s$ coincide.\\
This deformation can be done $ \pi_1M $-equivariantly since $ \pi $ is injective over $ N_t $ and we obtain the $ s=1 $ extremity $ \chext{h^1_t}{t}{t_0-\de}{t_1+\de} $ which induces the generic path of closed $ 1 $-forms $ \ch{\al_t'}{t} $ that we were searching for.\\

{\sc Proof of the claim:} The restriction of the projection $\pi$ to any individual stable or unstable manifold is injective. Moreover, we can suppose $\parent{\pi_1M\cdot\descT{S_t}{+}}\cap \im\parent{\lam_t^+}=\V$ by taking a smaller $\ve$ if necessary and hence, $ \pi|_{\im(I_t)} $ is injective for every $ t\in [t_0,t_1] $. Since $ \pi $ is also injective over any flow line, $ \pi|_{\im(I_t)} $ can also be supposed injective for $ t\notin [t_0,t_1] $. We deduce that for every time $ t $ we can consider a compact neighbourhood $ U_t $ of $ \im(I_t) $ where $ \pi $ is injective.\\

Since $ \formes $ is generic and contains a birth singularity of index $ 0 $ at $ t=t_0 $, its universal unfolding (see \cite[Ch. IV, \S6]{martinet}) is under the form $ \al_t(u,z)=d\parent{u^3-(t-t_0)u+Q(z)} $, where $ (u,z) $ are coordinates in some neighbourhood $ N\subset U_{t_0} $ of the birth zero $ s_{t_0}=c_{t_0} $ and $ t $ lives in $ [t_0^-,t_0^+] $ some small neighbourhood of $ t_0 $. By construction of the arcs $ I_t $, the path of functions $ f_t:[0,1]\to\bR $ is also generic with birth singularity at $ t=t_0 $. This allows one to find coordinates $ (u,z) $ splitting $ N $ as in the claim for $ t\in [t_0^-,t_0^+] $. We take $ N_t:=N $ for $ t\in [t_0^-,t_0^+] $. We find the neighbourhoods $ N_t $ for $ t\in [t_1^-,t_1^+] $ by a similar reasoning.\\

For $ t\in [t_0^+,t_1^-] $, let us take Morse models $ \cM(S_t)\subset U_t \supset\cM(C_t) $ of $ S_t,C_t $ with coordinates $ (u,z)\in\bR\x\bR^{n-1} $ such that \[ h_t|_{\cM(S_t)}(u,z)=h_t(S_t)-u^2+Q(z)\quad\text{and}\quad h_t|_{\cM(C_t)}(u,z)=h_t(C_t)+u^2+Q(z). \]
Further, we can require that $ (u,0) $ are coordinates of $ \im(I_t)\cap\cM(S_t) $ and of $  \im(I_t)\cap\cM(C_t) $ respectively, in such a way that the functions $ h_t|_{\cM(S_t)\cup\cM(C_t)} $ and $ h_t(I_t(u))+Q(z) $ coincide. Take $ N_t\subset U_t $ a tubular neighbourhood of $ \im(I_t) $ whose $ \bD^{n-1} $-fibers over $ \im(I_t)\cap\parent{\cM(S_t)\cup\cM(C_t)} $ are contained in the hypersurfaces $ \sing{u=\text{const}} $. We have then extended the $ z $-part of the $ (u,z) $-coordinates to the whole $ N_t $, verifying $ \im(I_t)=\sing{z=0} $.\\
By construction of $ I_t $, the functions $ h_t(u,z) $ and $ h_t(I_t(u))+Q(z) $ clearly coincide on $ \im(I_t)\prive\sing{S_t,C_t} $, where they do not have critical points. Their germs are thus isotopic and the path method of Moser (see \cite{moser}) can be used to find local isotopies that we glue by partition of unity; moreover, the isotopy can be chosen stationary at $ \cM(S_t)\cup\cM(C_t) $ since the functions coincided there before. This can be seen as a direct application of \cite[Lemma 3.1]{laudenbachReid}.\\
This can be continuously done in the whole interval $ [t_0^+,t_1^-] $; so we have found the mentioned coordinated tubular neighbourhoods for the intermediate times $ [t_0^+,t_1^-] $. Still, to have a coherent system of coordinate neighbourhoods all over the interval $ [t_0^-,t_1^+] $, we have to produce change of coordinates twice at $ t=t_0^+,t_1^- $ from the $ (u,z) $ coming respectively from the birth and elimination path in order to obtain the specified Morse models respectively around $ S_{t_0^+},C_{t_0^+} $ and $ S_{t_1^-},C_{t_1^-} $. The variables $ u $ and $ z $ are separated in both cases; to make the models coincide on $t=t_0^+$, it is enough to operate the coordinate change $\di u\mathop{\longmapsto}^{\vp} \sqrt{(u\pm \kappa)^3\mp 3\kappa(u\pm\kappa)^2} $ in the $ u $-part where $ \kappa:=\sqrt{\frac{t_0^+-t_0}{3}} $ and $ (+,-) $ corresponds to $ (S_t,C_t) $. Remark that the germs at $ 0 $ given by $ \Id $ and $ \vp $ are isotopic. A similar modification is done at $t=t_1^- $. 
\end{proof}

The next concept is useful to describe generic paths of closed $ 1 $-forms $ \formes $. It depends essentially on the choice of paths $ B_*(\al_t)\subset\wt{M} $ lifting the continuous paths determined by the zeroes $ Z(\al_t) $. Choose $ B_*(\al_0) $, liftings of the initial zeroes together with liftings of the birth zeroes; this determines $ B_*(\al_t) $, \emph{continuous liftings} of the zeroes.

\begin{definition}\label{df:CerfNov}
Let $ \formes $ be a generic path of closed $ 1 $-forms. Take a family of primitives $ \primits $ together with continuous liftings $ B_*(\al_t) $. Associated with this data, we have the Cerf-Novikov graphic, which is given by the set:
\[ \Gra(B_*):=\bigcup_{t\in[0,1]}\sing{t}\x h_t\parent{B_*(\al_t)}\subset [0,1]\x\bR.\]
\end{definition}

\begin{example}
The local change in the Cerf-Novikov graphic when we apply the elementary lips lemma for sinks \ref{lem:lips} can be seen in figure \ref{fig:LipsGraphic}.
\end{example}
\DibLocScalNomEtiq{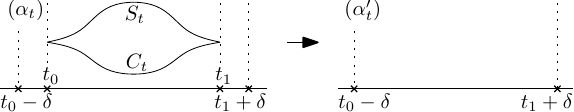}{h!}{1.1}{Elimination of a ``pair of lips''.}{fig:LipsGraphic}

\begin{lemma}[{\sc Elementary swallow-tail lemma for sinks}]\label{lem:SwallowTail}
Let $ (\al_t, \xi_t)_{\param} $ be generic. Suppose that $ \chext{s_t,s_t',c_t}{t}{t_0}{t_1} $ are continuous paths of saddles and sinks such that $ s_{t_0}=c_{t_0},s'_{t_1}=c_{t_1} $ are respectively birth and elimination and that there exists an $ \ve>0 $ 
and continuous families  $\ell_t^+\in\cL(s_t,c_t),(\ell'_t)^+\in\cL(s_t',c_t)$ for  $t\in (t_0,t_1)$ such that the  other separatrix of $\desc{s_t},\desc{s_t'}$ respectively verifies inequality \eqref{eq:LengthIneq} of lemma \ref{lem:lips}.\\
Then, for all $ \de>0 $ small enough, there exists a generic $ \ch{\al_t'}{t} $ such that:
\begin{enumerate}
\item $ \al_t=\al_t' $ for all $ t\notin (t_0-\de,t_1+\de) $,
\item $ Z(\al_t')=\bparent{Z(\al_t)\prive\sing{s_t,s_t',c_t}}\cup\sing{s''_t} $ for all $ t\in (t_0-\de,t_1+\de) $, where $ \chext{s_t''}{t}{t_0-\de}{t_1+\de} $ is a continuous path of saddles starting at $ s'_{t_0-\de} $ and ending at $ s_{t_1+\de} $.
\end{enumerate}
\end{lemma}

\begin{example}\label{exam:SwallowTailGraphic}
{\rm  Applying lemma \ref{lem:SwallowTail} affects the graphic as it is depicted in figure \ref{fig:SwallowTailGraphic}. The enwrapment of $ (\ell'_t)^+ $ does not vary in $ (t_0,t_1) $; denote it by $ g:=e\bparent{(\ell'_t)^+} $. Nothing guarantees a priori that $ g=1_{\pi_1M} $, or in other words, that the chosen liftings $ S_t' $ of $ s_t' $ are those realizing the elimination of the $ C_t $. The dashed line is just a translation by the vector $ \bparent{0,u(g^{-1})} $ of the curve $ h_t(S_t') $ of the graphic.
\DibLocScalNomEtiq{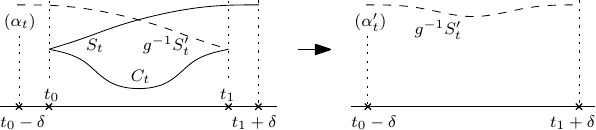}{h!}{1}{Elimination of a ``swallow-tail''.}{fig:SwallowTailGraphic}
}
\end{example}

The last tool that will be needed in order to prove the main theorem is lemma \ref{lem:ShiftLeft} below, which is again an adapted version of \cite[Lemma 2.5]{laudenbachReid} for closed $1$-forms. The remainder of this section supposes each closed $1 $-form $\alpha$ -- in a path or not -- having a distinguished primitive $h:\wt{M}\to\bR$.

\begin{definition}\label{df:AdaptCyl}
A cylinder adapted to $\al$ is a $\cC\cong\bD^{n-1}\x[-1,1]$ embedded into $\wt{M}$ verifying the next properties:
\begin{enumerate}
\item\label{prop:Inj} The restriction $\pi|_{\cC}$ is injective,
\item\label{prop:Caps} the \emph{caps} $\cC^\pm:=\bD^{n-1}\x\sing{\pm 1}$ are included on levels of $h$ and 
\item\label{prop:Lat} the restriction of $h$ to the \emph{lateral boundary} $\p_{\lat}\cC:=\p\bD^{n-1}\x [-1,1]$ has no critical point.
\end{enumerate}
We say moreover that such an adapted $\cC$ is \emph{semi-conjugated} to the function $F:\bR^n\to\bR$ if there exist an embedding $\vp:\cC\to\bR^n$ containing $0$ in its interior and $\psi\in\Diff(\bR)$ such that the following square commutes:
 \begin{equation}\label{eq:SemiConj}
\cuadr{\cC}{\bR}{\bR^{n}}{\bR}{h}{\vp}{\psi}{F} 
\end{equation}
\end{definition}

Fix an $i\in\sing{0,1,\ldots, n-1}$  and choose coordinates $x=(z,u)\in\bR^{n-1}\x\bR$ until the end of this section.

\begin{definition}\label{df:BirthPath}
A  path of closed $1$-forms is a \emph{birth path of index $i$ centred at time $t_0$} if there exists $\ve >0$ and a path of cylinders $\chext{\cC_t}{t}{t_0-\ve}{t_0+\ve}$ respectively adapted to $\al_t$ and semi-conjugated to  $F^{i,t_0}_{t}:\bR^n\to\bR$ where $\chext{F^{i,t_0}_t}{t}{t_0-\ve}{t_0+\ve}$ is a \emph{birth model of index $i$ centred at $t_0$}, namely:
\[
F^{i,t_0}_t(z,u):=\bparent{u^3-(t-t_0)u}+Q_i(z), \esp\text{where }Q_i\text{ is a quadratic form of index } i.
\]
We say that such a birth path starts at $(\al_{t_0-\ve},\cC_{t_0-\ve})$.
\end{definition}
Remark that $F^{i,t_0}_t$ has no critical point when $t<t_0$.

\begin{lemma}[{\sc Left-shifting of birth}]\label{lem:ShiftLeft}
Two statements hold:
\begin{enumerate}
\item Let  $\cC$ be a cylinder adapted to $\al$ and semi-conjugated to $F_0^{i,\ve}$.\\
Then there exists a birth path $\chext{\al_t}{t}{0}{2\ve}$ starting at $(\al,\cC)$.
\item\label{stat:Shift} Let  $\chext{\al_t}{t}{a}{b}$ be  a generic path with primitives $\chext{h_{t,0}}{t}{a}{b}$ and cylinders $\chext{\cC_{t,0}}{t}{a}{b}$ respectively adapted to $\al_t$ and semi-conjugated to $F_0^{i,\ve}$. Let $\chext{\be_{b,s}}{s}{0}{2\ve}$ be a birth path of index $i$ starting at $(\al_b,\cC_{b,0})$.\\
Then there exists a family $\chext{\be_{t,s}}{s}{0}{2\ve}$ of birth paths  of index $i$, parametrized by $t\in [a,b]$, extending the given birth path for $t=b$ and such that for every fixed $t\in [a,b]$, it starts at $(\al_t,\cC_{t,0})$.
\end{enumerate}
\end{lemma}
\begin{proof}
The first statement is easy to prove; we want to construct a family of primitives $\chext{h_t}{t}{0}{2\ve}$ starting at $h$, the given primitive of $\al$. Take first $\cC_t:=\cC$ and $\vp_t:=\vp$ for every $t\in [0,2\ve]$  and just choose $\chext{h_t}{t}{0}{2\ve}$ the functions making the diagram \eqref{eq:SemiConj} commute for $\psi=\Id$ and the model functions $\chext{F_t^{i,\ve}}{t}{0}{2\ve}$; then extend this path of functions $\pi_1M$-equivariantly over $\pi_1M\cdot\cC$. Of course, in order to have a honest path of primitives starting from $h$, we need the deformation to be stationary near the boundary of $\cC$; this can be obtained up to taking $\cC'$ a subcylinder of $\cC$ and a bump function $\omega$ with support contained on $\cC$, such that $\omega|_{\cC'}=1 $ and depending only on the $u$-coordinate of $\cC$. The diagram \eqref{eq:SemiConj} still commutes up to a suitable choice of $\chext{\vp_t}{t}{0}{2\ve}$ and rescaling $\chext{\psi_t}{t}{0}{2\ve}$ starting at $(\vp,\Id)$. The interested reader can compare this semi-conjugation method with that of \emph{elementary birth paths} in \cite[Ch.3, \S1]{cerfStrat}.\\
For the second statement, the square appearing on figure \ref{fig:TraceBirthShift} represents the two-parameter family $(\be_{t,s})$ that we want to construct: an extension of the given data which corresponds to the bottom and right-side edges of the mentioned square.\\
\DibLocScalNomEtiq{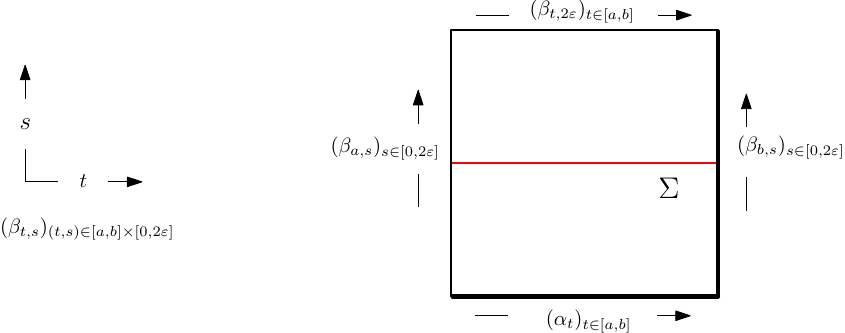}{h!}{0.7}{A path of birth paths. The birth parameters are denoted by $\Sigma$.}{fig:TraceBirthShift}

Remark that for every $t\in [a,b]$ the cylinder $\cC_{t,0}$ comes with a foliation induced by the levels of $\vp_{t,0}\co F^{i,\ve}_0$. Since this function has no critical point, the foliation is trivial and the leaves are all diffeomorphic to the lower level, which corresponds to the cap $\bD^{n-1}\x\sing{-1}$, the cylinders being adapted.\\
Take adapted subcylinders $\chext{\cC_{b,s}'}{s}{0}{2\ve}$ of those given by the birth path $\chext{\be_{b,s}}{s}{0}{2\ve}$ such that $\cC_{b,s}'$ contains on its interior the critical points of $F_s^{i,\ve}$ -- via the embeddings -- when $s\in [\ve,2\ve]$; by choosing a family of adapted subcylinders $\chext{\cC'_{t,0}\subset\cC_{t,0}}{t}{a}{b}$ ending at the selected $\cC'_{b,0}$, we construct a two-parameter family of diffeomorphisms $\bparent{\theta'_{t,s}:\cC'_{t,0}\to\cC'_{b,s}}_{(t,s)\in [a,b]\x[0,2\ve]}$ extending $\theta'_{b,0}:=\Id_{\cC'_{b,0}}$ and such that:
\begin{enumerate}
 \item for every $t\in [a,b]$, the map $\theta'_{t,0}:\cC'_{t,0}\to\cC'_{b,0}$ globally preserves the foliations and
 \item for every parameter value, $\theta'_{(t,s)}$ preserves the foliations near the boundary\footnote{Remark that for any $s\in [0,2\ve]$, the function $F_s^{i,\ve}$ has no critical point near $\p\cC'_{b,s}$.}.
 \end{enumerate}
Thanks to an extension of isotopies, this family can be promoted to a family of diffeomorphisms $\theta_{t,s}:\cC_{t,0}\to\cC_{b,s}$ preserving the foliation on the complementary part to the subcylinders. For every fixed $t\in [a,b]$, the maps $\chext{h_{t,s}}{s}{0}{2\ve}$ given by the composition $h_{b,s}\co\theta_{t,s}:\cC_{t,0}\to\bR$ start at $h_{t,0}$ and are semi-conjugated to the family $\chext{F_s^{i,\ve}}{s}{0}{2\ve}$ up to rescaling. Extend now $\pi_1M$-equivariantly the paths $\chext{h_{t,s}}{s}{0}{2\ve}$ over $\pi_1M\cdot\cC_{t,0}$ and by $h_{t,0}$ outside the mentioned union of cylinders. The two-parameter family $\bparent{h_{t,s}}_{(t,s)\in [a,b]\x[0,2\ve]}$ is made of primitives of the claimed path of birth paths $\chext{\chext{\be_{t,s}}{s}{0}{2\ve}}{t}{a}{b}$. 
\end{proof}

\section{One-parameter elimination of centers}\label{sect:Elim1}

This section is devoted to prove theorem \ref{th:ElimParam}. Let us denote by $\Sigma\subset (0,1)$ the set of \emph{stabilisation} -- birth and elimination -- times of a path  $\formes$ like in theorem \ref{th:ElimParam}. Denote by $ \nu\geq 0 $ the amount of times $ t $ where $ \al_t $ is of type birth of index $ 0 $. As the extremities have no center, we have necessarily $ \nu $ elimination times of index $ 0 $. If $ \nu=0 $ we are done because the extremities of $ \formes $ are supposed to have no center.

\begin{definition}\label{df:NormalPath}
A generic path of closed $1$-forms is said to be \emph{normal} if the set $\Sigma$ of stabilisation times is ordered as follows:
\begin{enumerate}
\item first births of index $i>0$, then births of index $0$ noted by  $\sing{a_1<\ldots <a_{\nu}}$, then eliminations of index $0$ noted by $\sing{b_{\nu}<\ldots <b_1}$, finally eliminations of index $i>0$. Moreover, for every $i=1,\ldots, \nu$:
\item the continuous path of sinks $(c_t)_t$ which starts at $t=a_i$ ends  at $t=b_i$.
\end{enumerate}
\end{definition}

\begin{proposition}\label{pro:Shift}
If $\dim(M)>1$, a birth time $t_0\in \Sigma$ of a generic path $\formes$ can be shifted to the left as much as we want. More precisely:\\
For every $t_0'\in (0,t_0)$ there exists an $\ve>0$ and a generic path $\chext{\al_t'}{t}{0}{1}$  such that:
\begin{enumerate}
\item $\al_t'=\al_t$ for every $t\notin [t_0'-\ve,t_0+\ve]$,
\item $\Sigma'=\bparent{\Sigma\prive\sing{t_0}}\cup\sing{t_0'}$ and 
\item the birth singularities of $\al_{t_0},\al'_{t_0'}$ have same index.
\end{enumerate}
\end{proposition}

Proposition \ref{pro:Shift}, together with a completely symmetric version for \emph{right-shifting of eliminations}, allow us to restrict the attention to normal paths in an evident manner.
\begin{proof}[Proof of proposition \ref{pro:Shift}]
Note by $p_{t_0}\in Z(\al_{t_0})$ the considered birth zero and choose a lifting $P_{t_0}$. If $i$ denotes the index of $p_{t_0}$, the birth model $\chext{F_t^{i,t_0}}{t}{t_0-\ve}{t_0+\ve} $ of index $i$ constitutes a path of primitives for some $\ve>0$ and some coordinates $(u,z)$ on a neighbourhood $N$ of $P_{t_0}$ centred on $0\in\bR^n$. Over $N$, the maps $\chext{F_t^{i,t_0}}{t}{t_0-\ve}{t_0+\ve}$ are semi-conjugated to themselves. Let $\de>0$ such that the critical values of $F_t^{i,t_0}$ are included on $[-\de,\de]$, and this for every $t\in [t_0-\ve,t_0+\ve]$. We easily find cylinders $ \chext{\cC_t}{t}{t_0-\ve}{t_0+\ve}$ such that for every $t\in [t_0-\ve,t_0+\ve]$:
\begin{enumerate}
\item   $\Crit(F_{t}^{i,t_0})\subset \Int{\cC_t}, \cC_t\subset F^{-1}\bparent{[-\de,\de]}\cap N$ and
\item properties \ref{prop:Caps} and \ref{prop:Lat} of definition \ref{df:AdaptCyl} are verified.
\end{enumerate}
Since $\de$ can be chosen as small as desired up to taking a smaller $\ve$, property \ref{prop:Inj} concerning the injectivity of the projection can be also achieved. The path $\chext{\al_t}{t}{t_0-\ve}{t_0+\ve}$ is so a birth path of index $i$. Of course we can suppose $\ve < t_0'$. Set $a:=t_0'-\ve, b:=t_0-\ve$.\\

Consider $\primits$ extending the choice of primitives that we have made near $t_0$ and the $1$-dimensional submanifold of $[0,1]\x \wt{M}$ given by:
\[\cT:=\ens{(t,Q_t)\in [0,1]\x \wt{M}}{Q_t\in\Crit(h_t)}.\]
The embedded curve $\cT$ cannot disconnect $[0,1]\x \wt{M}$ since its dimension  is $n+1>2$.  We choose two points on the connected manifold $\bparent{[0,1]\x \wt{M}}\prive\cT$:
\begin{itemize}
 \item $K:=(a,x_a)$ where $x_a$ is a regular point of $h_a$ and
 \item $L:=(b,x_b)$, such that $x_b\in\Int\cC_{b}$. 
 \end{itemize} 
These two points can be connected by an arc 
\begin{equation}\label{eq:ArcShift}
\aplic{\gamma}{[a,b]}{\bparent{[0,1]\x \wt{M}}\prive\cT}{t}{(t,x_t)};
\end{equation}
remark that any two functions without critical points are semi-conjugated. Each time that we have $x$ a regular point of a primitive $h$ of $\al$, we find a cylinder $\cC$ adapted to $\al$, containing $x$ in its interior and such that $h|_{\cC}$ has no critical points; this cylinder is therefore semi-conjugated to $F_0^{i,\ve}$, since the initial extremity of a birth model of index $i$ has no critical points. We can therefore choose a continuous family of cylinders $\chext{\cC_{t,0}}{t}{a}{b}$ along $\gamma$, respectively adapted to $\al_t$ and being semi-conjugated to $F_0^{i,\ve}$. Apply then statement \ref{stat:Shift} of lemma \ref{lem:ShiftLeft} to the next data:
\begin{itemize}
\item the path $\chext{\al_t,\cC_{t,0}}{t}{a}{b}$ and
\item the birth path given by $\chext{\be_{b,s}}{s}{0}{2\ve}$ where $\be_{b,s}:=\al_{b+s}$ for every $s\in[0,2\ve]$.
\end{itemize}
We obtain a path of birth paths $\chext{\be_{t,s}}{s}{0}{2\ve}$ depending on $t\in [a,b]$, each starting from $\al_t$. Consider then the one-parameter family given by: \[\al'_t=\sistIVU{\al_t}{\text{if } t\in [0,a]}{\be_{a,t-a}}{\text{if } t\in [a,a+2\ve]}{\be_{t-2\ve,2\ve}}{\text{if } t\in [a+2\ve,b+2\ve]}{\al_t}{\text{if } t\in [b+2\ve,1]}. \] 
This new path is homotopic, relative to its extremities, to the original one: the homotopy is stationary for $t\notin [a,b+2\ve]$ and $\chext{\al_t'}{t}{a}{b+2\ve}$ is nothing but the left-side and top edges of the square depicting the two-parameter family $(\be_{t,s})$ on figure \ref{fig:TraceBirthShift}. The path $\chext{\al_t'}{t}{a}{b+2\ve}$ clearly crosses the birth stratum at time $a+\ve=t_0'$ while $b+\ve=t_0$ is no more a birth time. We find an explicit homotopy of paths by sliding the birth path to the left as figure \ref{fig:BirthShift} suggests.
\DibLocScalNomEtiq{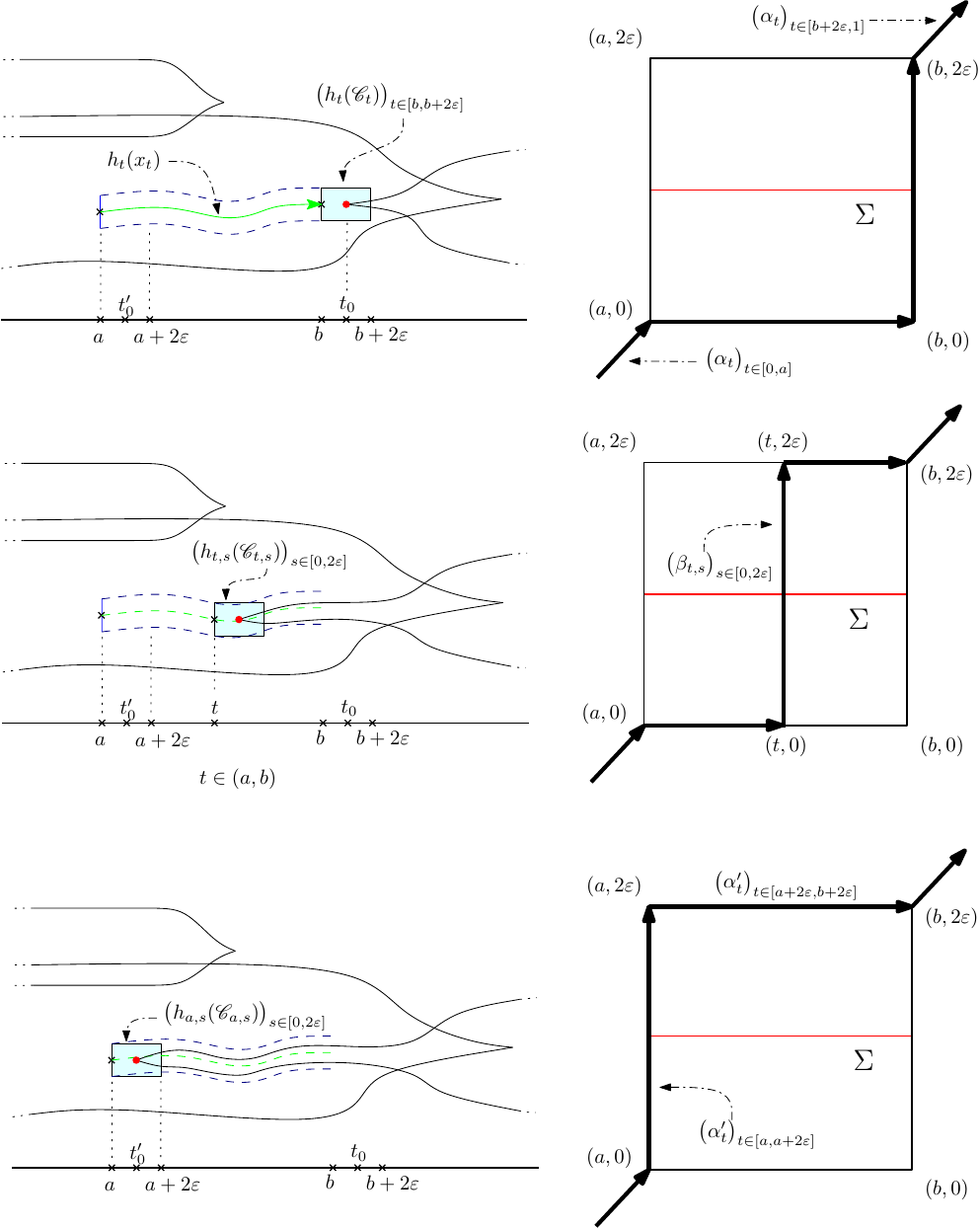}{h!}{.625}{Effect of a birth shift on the Cerf-Novikov graphic.}{fig:BirthShift} 
 
\end{proof}

From now on, we concentrate on eliminating the more internal path of sinks $\chext{c_t}{t}{t_0}{t_1}$ of a normal path $\formes$: the interval $ (t_0,t_1) $ is  made up of Morse times. We would be done if there was a path of saddles $\chext{s_t}{t}{t_0}{t_1}$ satisfying the hypothesis of lemma \ref{lem:lips} with respect to the specified sinks. We have defined connecting saddles in order to find such a candidate of saddles path $\chext{s_t}{t}{t_0}{t_1}$. The connecting saddle of $c_t$ is determined near $t_0$ and $t_1$ by genericness of $\formes$, as we have seen in the proof of the claim of lemma \ref{lem:lips}. However, lemma \ref{lem:ConnectSadd} does not apply to the Morse family $\chext{\al_t}{t}{t_0^+}{t_1^-}$ since we cannot equip it with pseudo-gradients $\chext{\xi_t}{t}{t_0^+}{t_1^-}$ being Morse-Smale at every time: the set of times $t$ where $\xi_t$ is Morse-Smale, is only a dense subset of $[t_0^+,t_1^-]$.\\
We want a more general condition than Morse-Smale's one, which ensures existence of connecting saddles for generic paths. The next proposition goes in this direction.

\begin{proposition}\label{pro:GenConSadd}
Let $ \formes $ be a generic path of Morse closed $ 1 $-forms together with a path $ \ch{c_t}{t} $ of sinks. For every generic path $ \champs $ of  pseudo-gradients of $ \formes $, consider $$ \ensdefTres{\cT}{t\in [0,1]}{c_t \text{ admits a connecting saddle } s_t }; $$ 

then the map $ \cT\flee{L}\bR^+ $ given by $t\mapsto L(\ell_t^+)$ is bounded by some constant $K>0$.
\end{proposition} 
\begin{proof}
Since $\champs  $ is generic, the set $ \cT $ is at least dense in $ [0,1] $ by lemma \ref{lem:ConnectSadd}. If the map $ L $ of the statement is not bounded, there exists $ (t_i)\subset \cT  $ converging to some $ \tau\notin \cT $ such that the sequence $ \bparent{L(t_i)}_{t_i\in\cT} $ diverges. Let us see that this can not happen.\\
Take $ \tau\notin \cT $ and $ \eta_{\tau} $, a Morse-Smale pseudo-gradient for $ \al_{\tau} $. Consider $ \Omega $ a small Morse neighbourhood of $ Z(\al_{\tau}) $. We can find an isotopy $ (\vp_t)_t $ of $ \varphi_{\tau}:=\Id_M $ such that $ \vp_t^*(\al_t)|_{\Omega}=\al_{\tau}|_{\Omega} $ for $ t $ near $ \tau $. Take a $ \de>0 $ small enough such that the condition $ \scal{ \vp_t^*(\al_t)}{\eta_{\tau}}|_{M\prive\Omega}<0 $ holds for every $ t $ such that $ \abs{t-\tau}<\de $. The vector field $ \eta_t:=(\phi_t)_*(\eta_{\tau}) $ is still a pseudo-gradient for every $  t\in [\tau-\de,\tau+\de] $ and clearly Morse-Smale since $ \phi_t $ is a diffeomorphism. In particular, the sinks $ c_t $ have connecting saddles relative to $ \chext{\eta_t}{t}{\tau-\de}{\tau+\de} $, whose connecting orbits verify $ L(\ell^+_{\eta_t})<K$ for every $ t\in[\tau-\de,\tau+\de] $ and some $ K>0 $ since $ \eta_t $ is Morse-Smale in these times.\\

Consider now any Morse $ \al $ together with a primitive $ h:\wt{M}\to\bR $ and $ \xi $ a Morse-Smale pseudo-gradient for $ \al $. If $ C $ is a lifting of a sink $ c $ of $ \al $, define for all $ a>h(C) $, the basin $ \Bas(a)\subset \wt{M} $ given by the closure of $ \asc{C}\cap h^{-1}\parent{(-\infty,a)} $. If $ s $ is a connecting saddle for $ c $ of connecting orbit $ \ell^+ $, call $ S $ the initial extremity of the lifting $ \wt{\ell^+} $ going to $ C $. Remark that for every $ \ve>0 $, any critical point of $ h $ of index $ 1 $ contained in $ \Bas(h(S)-\ve)  $ corresponds to a saddle in $ \cS_c^{\ker} $. In particular we have that $ h(S) $ coincides with the value \[ D:=\sup\ens{a\in\bR}{\al|_{\pi(\Bas (a))}\text{ is exact}}. \]
The latter value only depends on $ \al $, and the length of the connecting orbit for any Morse-Smale pseudo-gradient $ \xi $ coincides with $ D-h(C) $. We conclude that $ L(\ell^+_{\xi_t})=L(\ell^+_{\eta_t}) $ for all $ t\in [\tau-\de,\tau+\de] $ where $ \xi_t $ is Morse-Smale. 
\end{proof}

We need the notion of truncated unstable manifold associated with some pseudo-gradient $ \xi $: denote by $ \ga^p_x $ the portion of orbit going from $ p\in Z(\al) $ to some $ x\in\desc{p} $. For every $ K>0 $, we set \[ \descT{p}{K}:=\ens{x\in\desc{p}}{L(\ga^p_x)\leq K}. \]

\begin{definition}\label{df:KTransv}
Let $K>0$. We say that a pseudo-gradient $ \xi $ for a closed $1$-form $\al$ is $K$-transversal if \[
\descT{p_{t}}{K}\trans\asc{q_t},\text{ for every } p_t,q_t\in Z_*(\al_t). \]
\end{definition}

Reconsider the Morse path $\chext{\al_t}{t}{t_0^+}{t_1^-}$ as before proposition \ref{pro:GenConSadd}, which says that the length $ L $ of connecting orbits do not explode in this family. Take $K>K'>0$ where $ K $ is a bounding value for $L$. So, any sink $c_t$ in our Morse family, admits a connecting saddle on times when the equipping $\xi_t$ verifies the $K$-transversality property.\\
A generic equipment $\chext{\xi_t}{t}{t_0^+}{t_1^-}$ verifies the $K$-transversality property everywhere but in a finite set of times where a \emph{$K$-sliding} appears. Briefly explained, this results from the openness and density of the condition of $K$-transversality for pseudo-gradients: the truncated unstable manifolds admit compactifications $\adh{\descT{p_{t}}{K}}$ to submanifolds with corners of $M$ (as in \cite[Prop. 2.11.2]{latour}); the mentioned result is obtained by  
applying Thom's transversality theorem (see also \cite{hirsch}) to the finite collection of varying families $ \cup_{t}\adh{\descT{p_{t}}{K}},\cup_t\asc{q_t}$. More precisely, the transversality theorem is applied to a moving holed sphere $\Delta(P_t)$ of dimension $i-1$ with respect another fixed one $\Delta(gQ_t)$ of dimension $n-j-1$, both embedded in a $n-1$ dimensional manifold $L$, which corresponds to a level of $h_t$ between $gQ_t$ and $P_t$. The mentioned holed\footnote{Holes appear when the unstable manifold crosses a level containing a critical point of $h_t$ in its closure.} spheres are the intersection with $L$ of $ \adh{\descT{P_{t}}{K}}$ and $\asc{gQ_t}$ respectively, $i,j$ are respectively the index of $p_t,q_t$ and the possible values of $g$ are $u$-bounded by the truncation by $K>0$.\\ 
We deduce that $K$-transversality fails in a finite set of times where we find a unique orbit  $ \gli $ connecting two not necessarily different zeroes of same index $i$. This is properly proven in \cite[Prop. 2.2.33]{yoTesis}; a geometrically instructive description of the impact of this accident on the Novikov maps induced by $\xi_t$ can be found in \cite[Prop. 2.2.36]{yoTesis}, but consulting \cite[Th. 7.6 - Basis theorem]{milnorHcob} should suffice since it explains the analogue case of real-valued equipped Morse functions.\\

 The property of $ K $-transversality is so enough to ensure the existence of connecting saddles, while $ 0 $-excellence ensures uniqueness. But, for our generic $\chext{\al_t}{t}{t_0^+}{t_1^-}$,  $0$-excellence also holds everywhere but in a finite set of times, that we call \emph{competition} times. This comes from the work of  \cite{cerfStrat}: a generic path of functions $ f_t:V\to\bR $ from a compact manifold $ V $ has a finite set of times $ \tau_i $ where two critical points $ p_{\tau_i},q_{\tau_i} $ cross their critical value. This holds so for the path $ h_t|_V $ where $ \chext{h_t}{t}{t_0^+}{t_1^-} $ are primitives of $ \chext{\pi^*\al_t}{t}{t_0^+}{t_1^-} $ and $ V $ is a compact neighbourhood containing $ \bigcup_{t\in[t_0^+,t_1^-]}\parent{\asc{C_t}\cap h^{-1}([h_t(C_t),h_t(C_t)+K]} $ for a continuous lifting $\chext{C_t}{t}{t_0^+}{t_1^-}$ of the sinks.\\  
 
In this setting, call $ \cR $ the \emph{finite set} of times where $ K $-transversality of $ \xi_t $ or $ 0 $-excellence of $ \al_t $  fails. The map $[t_0^+,t_1^-]\prive\cR\flee{s}M$ defined by $ t\mapsto s_t $, where $s_t$ is the unique connecting saddle $ s_t $ for $ c_t $, is continuous. Let us study what happens near each $\tau\in\cR$.\\

\underline{\nth{1} problem - {\sc Competition:}}\, $\al_{\tau}$ is not $0$-excellent for $\xi_{\tau}$. We call $\tau$ a \emph{competition time}, where two saddles $ s^1_{\tau},s^2_{\tau} $ compete to be the connecting saddle of $ c_{\tau} $. The map $ s $ presents a discontinuity as figure \ref{fig:Competition} suggests.

\DibLocScalNomEtiq{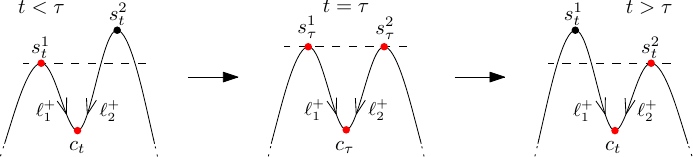}{h!}{1}{Two normal saddles competing on $ t=\tau $.}{fig:Competition}

We can compare competition times with the fact that a generic path of functions $ f_t:V\to\bR $ from a compact manifold $ V $ has a finite set of times $ \tau_i $ where two critical points $ p_{\tau_i},q_{\tau_i} $ cross their critical value (see \cite{cerfStrat}). This holds so for the path $ h_t|_V $ where $ \primits $ are primitives of $ \ch{\pi^*\al_t}{t} $ and $ V $ is a compact neighbourhood containing $ \bigcup_{t\in[t_0^+,t_1^-]}\parent{\asc{C_t}\cap h^{-1}([h_t(C_t),h_t(C_t)+K]} $ for a continuous lifting $\chext{C_t}{t}{t_0^+}{t_1^-}$ of the sinks.\\

\underline{\nth{2} problem: - {\sc $K$-sliding:}} As we have mentioned, in these times $\xi_{\tau}$ has a unique orbit  $ \gli \in\cL(s_{\tau},s'_{\tau})$ connecting two not necessarily different zeroes of same index $j$. The only case which affects connecting saddles is that of $j=1$ and $s_t$ being the connecting saddle for $c_t$ on times $t<\tau$.\\
~\newline
\begin{minipage}{0.6\textwidth}
This forces the saddle $s_{\tau}'$ to belong to $\cS_{c_\tau}^{\ker}$: the accident does not concern $\desc{s_t'}$ and $s_t'$ conserves its type for $t$ near $\tau$; if $s_{\tau}'$ is a normal saddle, $s_t$ cannot be the connecting saddle for $c_t$ when $t<\tau$ because the connecting orbit $(\ell^+_t)'$ of $s_t'$  would be shorter than the connecting orbit $\ell^+_t$ of $s_t$.\\
The connecting orbits $ (\ell_t^+)$  converge for $t\flee{t<\tau} \tau$ to a broken orbit $\gli*\ell_{\tau}' \in \cL(s_{\tau},s'_{\tau})*\cL(s'_{\tau},c_{\tau})$ as the figure \ref{fig:Vesube} suggests.\\
In particular, the enwrapments of the connecting orbits $ \ell_t^+ $ for times just before and after $ t=\tau $ are related by some $ g\in\ker(u) $: the length of the connecting orbit $ \ell^+_{t} $ presents an avoidable discontinuity at $ t=\tau $.
\end{minipage}
\hspace{5pt}
\vrule
\begin{minipage}{0.35\textwidth}
\centering
\includegraphics[scale=1]{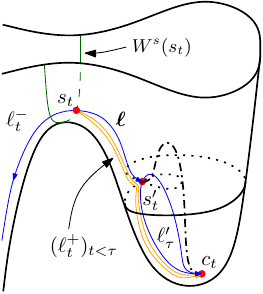}
\captionof{figure}{$K$-sliding time.}
\label{fig:Vesube}
\end{minipage}\\
\newline

We describe now a local modification of a generic path which allows one to transmute a $ K $-sliding time into a pair of competition times.

\begin{proposition}\label{pro:KSlideToCompet}
Let $ \ch{s_t,c_t,\ell_t^+}{t} $ be as the above description, associated with a path of Morse closed $ 1 $-forms $ \formes $ provided with a generic equipment $ \champs $ with no competition times and only one $ K $-sliding time at $ t=\tau $. There exists a deformation to a generic couple $ \ch{\al_t',\xi_t'}{t} $ such that:
\begin{enumerate}
\item nothing has changed in the complementary  of some interval $ (t_0,t_1) $ containing $ \tau $,
\item the only accidents of the equipment $ \chext{\xi_t'}{t}{t_0}{t_1} $ are two competition times $ \tau^-<\tau^+ $, 
\item the sinks of $ \al_t$ and $\al_t' $ are the same for every time $ t $.
\end{enumerate}
\end{proposition}
\begin{proof}
Choose initial liftings $ B_*(\al_0) $ of the zeroes and take $ B_*(\al_t) $ the continuous path of liftings associated with it. We denote by $ \ch{C_t}{t} $ the lifting of the sinks $ \ch{c_t}{t} $ containing $ c_{\tau} $, the sink involved with the $ K $-sliding phenomenon. The stated deformation is going to produce the local changes in the Cerf-Novikov graphic as it is depicted in figure \ref{fig:SlideToCros}.

\DibLocScalNomEtiq{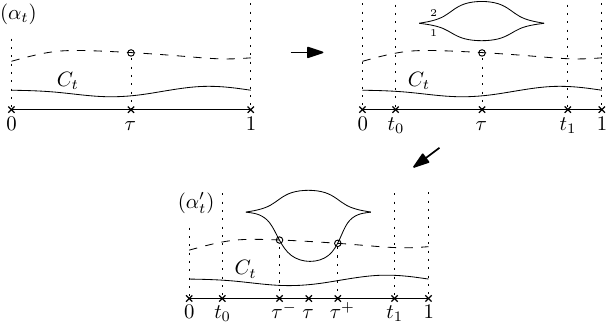}{h!}{1}{Changing a $ K $-sliding into two competitions.}{fig:SlideToCros}
The connecting orbits $ \ell_t^+ $ lift uniquely to orbits $ \wt{\ell_t^+} $ going to $ C_t $, and starting from $ g^{-1}S_t  $ where $ g $ is the enwrapment of $ \ell_t $ and $ S_t $ is the chosen lifting of the connecting saddle $ s_t $ of $ c_t $; the dashed line represents the values $ h_t(g^{-1}S_t) $, a translated curve of the graphic.\\

There exists an $ \ve >0 $ small enough such that the intersection of $ \asc{S_t} $ with $ L_t $, the level of $ h_t(S_t)+\ve $, is a $ (n-2) $-sphere for all $ t $ near $ \tau $, which we denote by $ \bS_t$. Take $ K_t $ a relatively compact open neighbourhood of $ \bS_t $ in $ \wt{M} $ such that $ \pi |_{\adh{K_t}} $ is injective.\\
For every $ t $ in an interval $ [t_0,t_1] $ containing $ \tau $, choose an arc $ I_t:[0,1]\to L_t\cap K_t $ intersecting $ \bS_t $ transversely only once at time $ \theta=\frac{1}{2} $. By the hypothesis on the equipment $ \champs $, one of the connected components of $ I_t\bparent{[0,1]\prive\sing{\frac{1}{2}}} $, say $ I_t^+:=I_t\parent{(\frac{1}{2},1]} $, is entirely contained in $ \asc{gC_t}\cup\asc{g'S_t'} $, where $ g' $ is the enwrapment of the accidental orbit $ \gli\in \cL(s_{\tau},s'_{\tau}) $. These arcs can be chosen in order to have $ I_t^+\trans\asc{g'S_t'} $; in particular, the extremities $ I_t(1) $ can be taken into $ \asc{gC_t} $ for every $t\in [t_0,t_1]$. The other component $I_t^- $ verifies 
\[
I_t^- \cap\bigcup_{u(k)\geq 0}\asc{kgC_t}=\V :
\]
if there was not the case, $ S_t $ would be in $ \adh{\asc{kgC_t}} $ leading to a contradiction with the fact that $ s_t $ is a connecting saddle for $ c_t $ whose connecting orbit has enwrapment $ g $.
Again by a transversality argument together with \eqref{eq:PairConnecting}, we can take the other extremities going to a sink for every $t\in [t_0,t_1]$; more precisely:
\begin{equation}\label{eq:PairConnecting}
I_t(0)\in \asc{k'gC_t'}\esp\text{where}\esp u(k')<0\esp\text{or}\esp\pi(C_t')\neq \pi(C_t).
\end{equation}

For every $ t \in [t_0,t_1] $, we obtain new primitives $ \chext{h_t'}{t}{t_0}{t_1} $ by modifying $ \pi_1M $-equivariantly the initial $ h_t $ in $ K_t\cap h_t^{-1}\bparent{(h_t(S_t),+\infty)}  $ by introducing a cancelling pair of critical points $ S_t'',R_t $ of respective index $ 1,2 $. The new generic family of pseudo-gradients, which only differs from the original on $ \pi(K_t) $, can be chosen such that $ \desc{S_t''}\cap L_t=\sing{I_t(0),I_t(1)} $ and $\desc{R_t}\cap L_t=I_t\bparent{(0,1)}  $: a new pair of zeroes of index $ 1,2 $ appears now in times $ t\in(t_0+\de,t_1-\de) $ for a small $ \de $. For the new birth and elimination times $t=t_0+\de,t_1-\de$, the unstable manifold of $R_t=S_t''$ is a half $2$-disk which intersects $L_t$ precisely at $I_t\bparent{[0,1]}$. The associated Cerf-Novikov looks as in the second picture of figure \ref{fig:SlideToCros}. From genericness of $ \chext{\xi_t'}{t}{t_0}{t_1} $ together with \eqref{eq:PairConnecting}, we deduce that the new saddles $ s''_t:=\pi(S_t'') $ belong indeed to $ \cS_{c_t}^n $, and this for every $ t\in (t_0,t_1) $.\\

The separatrix of $ S_t'' $ passing through $I_t(1)$ evidently crosses the level of $ S_t $ since it goes to $gC_t$ for every $ t\in (t_0,t_1) $; the other separatrix of $ S_t'' $, passing through $I_t(0)$ goes to $k'gC_t'$ and \eqref{eq:PairConnecting} allows us to decrease the value of $k'gC_t'$ under the level of $S_t$ in the case\footnote{The other case says that $k'gC_t'=k'gC_t$ is already under $gC_t$.} $\pi(C_t')\neq\pi(C_t)$ by applying lemma \ref{lem:Arrang} with parameters. Since $C_t'$ is a sink, we can decrease its value even more: below the level of $gC_t$, this will be important in the Case $2$ of the proof of theorem \ref{th:ElimParam}. We can so apply the rearrangement lemma \ref{lem:Arrang} to the one-parameter family of $ S_t'' $ to continuously decrease the value of $ S_t'' $ under $ h_t(S_t) $ for $ t $ near $ \tau $. The saddle $ s_t'' $ becomes thus the connecting saddle for $ c_t $ into the interval $ (\tau^-,\tau^+) $ of times $ t $ where $ h_t'(S_t'')<h_t'(S_t) $. Times $ t=\tau^-,\tau^+ $ corresponds to competition times between $ s_t $ and $ s_t'' $. 
\end{proof}

We prove now our main result, which is theorem \ref{th:ElimParam} announced in the introduction. 

\begin{proof}[\textbf{Proof of theorem \ref{th:ElimParam}}]
We only eliminate the sinks; the sources can be treated similarly.\\
As we have argued in the beginning of this section, we can suppose that $\formes $ is normal. Denote by $ a:=a_{\nu} $ the last birth time and by $ b:=b_{\nu} $ the first elimination time of $ \formes $. We are going to produce a generic path $ \ch{\al_t^{(1)}}{t} $ such that:
\begin{enumerate}
\item it coincides with the original $\formes  $ outside of $ (a-\ve,b+\ve) $,
\item $ \ch{\al_t^{(1)}}{t} $ has $ \nu-1 $ birth times of index $ 0 $.
\end{enumerate}
The path $ \ch{\al_t^{(1)}}{t} $ clearly share extremities with $ \formes $. Iterating the construction $ \nu $ times, we obtain a generic $ \ch{\al_t^{(\nu)}}{t} $ which is the announced $\ch{\be_t}{t} $.\\

Choose continuous liftings $ B_*(\al_t) $ and primitives $ \primits $ of $ \formes $. Consider the path of local minima of $ \primits $ given by $ \chext{C_t}{t}{a}{b} $, the lifting of the corresponding path of sinks $ \chext{c_t}{t}{a}{b} $. Take $ \champs $ a generic equipment for $ \formes $. We replace the finite amount of $ K $-sliding times in the interval $ [a,b] $ concerning our sinks by twice competition times by introducing trivial pairs of index $ (1,2) $ as in proposition \ref{pro:KSlideToCompet}. The dimension hypothesis is used here: the new zeroes of index $ 2 $ are not sources because $ n\geq 3 $. Remark that after this operation, our path -- still denoted $\formes$ -- is no more normal: the Cerf-Novikov graphic over the interval  $[a,b]$ contains finitely many times the last picture of figure \ref{fig:SlideToCros}, each of those contributing with four accidents. Let us denote by $ \sing{t_i}_{i=1}^r\subset (a,b) $ the union of the finite set of competition times related to $ \chext{c_t}{t}{a}{b} $ together with the times of the new births and eliminations of index $1$.\\

Suppose first that the simplest case, $ r=0 $ holds. Since there is not any competition, $K$-sliding or stabilisation time between $ a $ and $ b $, we have a continuous path of connecting saddles $ \chext{s_t}{t}{a}{b} $ for our sinks. Consider the set of times $\Delta\subset (a,b)$ where the family of non-connecting separatrices $\chext{\ell_t^-}{t}{a}{b}$ do not verify $L(\ell_t^-)> L(\ell_t^+)$. For any $t\in \Delta, s_t$ is the connecting saddle for $c_t$ and  $\ell_t^-$ has to go to a sink $c_t'\neq c_t$; we can so apply the one-parameter version of the rearrangement lemma \ref{lem:Arrang} to $c_t'$ for $t$ in any interval containing $\Delta$ in order to have $L(\ell_t^-)>L(\ell_t^+)$ for all $t\in (a,b)$. The hypothesis of the elementary lips lemma for sinks \ref{lem:lips} respectively to the couples $\chext{s_t,c_t}{t}{a}{b}$ is now verified and we obtain the claimed family $ \ch{\al_t^{(1)}}{t} $ by applying it; the graphic changes as in figure \ref{fig:LipsGraphic}.\\

Let us treat now the case $r=1$. It means that there is only one competition at $t=t_1$: if $t_1$ was a birth time, necessarily $r\geq 4$. Since there is no accident in $ (a,t_1) $, the saddle $ s_t $ is \emph{the} connecting saddle of $ c_t $ for every $t\in (a,t_1)$. In $ t=t_1 $, a connecting saddle $ s'_{t_1}\neq s_{t_1}  $ competes with $ s_{t_1} $ and  $s_t'$ is then \emph{the} connecting saddle for $ c_t $ when  $t\in (t_1,b)$. Since there is no accident in the interval $(a,b)$ other than the competition at $t=t_1$, the continuous path of saddles $ (s_t)_t $ starting at $ s_a=c_a $ is defined until $ t=b+\ve$ for some small $\ve>0 $: the saddles $s_t$ and $s_t'$ are different for every $t\in (a,b)$ and $s_t$ does not cancel at $t=b$. In the same way, the continuous family $ (s'_t) $ is defined from $ a-\ve $. 

Competitions do not change the enwrapment of orbits. In particular, we can suppose that the length of the non-connecting orbit of $ \desc{s_t} $ is bigger than $ L(\ell^+_t) $ for every time $t$ since it is the case for times near $ a $. The same reasoning applies to the non-connecting orbit of $ \desc{s'_t} $. We can apply the elementary swallow-tail lemma for sinks \ref{lem:SwallowTail} to $ \ch{\al_t}{t} $, relative to the zeroes $ \chext{s_t,s_t',c_t}{t}{a}{b} $. The graphic has changed as in figure \ref{fig:SwallowTailGraphic} and we obtain the claimed  $ \ch{\al_t^{(1)}}{t} $.\\

The rest of the proof consists in obtaining a path $\ch{\al_t''}{t}$ with same extremities and births of index $0$ than the previous one but where the value of $r$ has been lowered. Three different cases can occur:\\

\underline{Case 1.A: $t_1,t_2$ are competition times.} We perform a preliminary modification to $ \formes $ as figure \ref{fig:ElimPrelimCompV2} suggests.
\DibLocScalNomEtiq{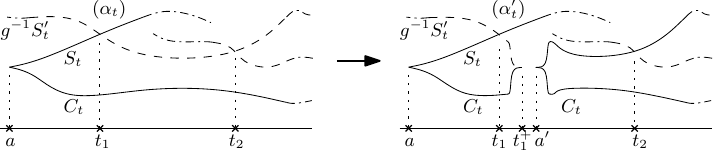}{h!}{1}{Case 1.A - Effect in the Cerf-Novikov graphic of the preliminary modification on $ \formes $.}{fig:ElimPrelimCompV2}

The modified $ \ch{\al_t'}{t} $ is a generic path of $ 1 $-forms obtained by inserting a loop $ (\ga^{1}_t)$ in $ \formes $, based at time $t_1+\ve $ for an $ \ve>0 $ small enough. This loop is constructed by following forwards then backwards the path realizing the elimination of $ s_{t_1+\ve} $ with $ c_{t_1+\ve} $ given by lemma \ref{lem:Elim0}. The new birth and elimination times of the considered path of sinks $ (c_t)_t $ are denoted by: $ t_1<t_1^+<a'<t_2 $.\\
~\newline
\begin{minipage}{0.5\textwidth}
Remark that the zeroes $ \chext{s_t,s_t',c_t}{t}{a}{t_1^+} $ of the    path $\ch{\al_t'}{t}$ are as in the situation $r=1$ that we have just described. We apply lemma \ref{lem:SwallowTail} to the mentioned zeroes to obtain $ \ch{\al_t''}{t} $, whose graphic is depicted in figure \ref{fig:ElimCompFinal}: we have reduced by one the number of competitions.
\end{minipage}
\hspace{3pt}
\vrule
\hspace{3pt}
\begin{minipage}{0.45\textwidth}
\centering
\includegraphics[scale=1]{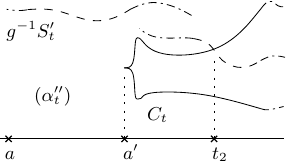}
\captionof{figure}{One competition less.}
\label{fig:ElimCompFinal}
\end{minipage}\\
\newline

\underline{Case 1.B: $t_1 $ is competition, $t_2$ is birth.} This case can be treated similarly to the latter one: the birth singularity comes from the application of proposition \ref{pro:KSlideToCompet} and the graphic of the path that we are considering is as in the first picture of figure \ref{fig:Case1B}. By inserting a loop at $t_1+\ve$ in the same way that we have done in the case 1.A, we manage to isolate a situation verifying the hypothesis of lemma \ref{lem:SwallowTail} near the competition at $t_1$ as in the second picture of figure \ref{fig:Case1B}; we apply lemma \ref{lem:SwallowTail} to obtain the claimed path $ \ch{\al_t''}{t} $ having one competition less.
\DibLocScalNomEtiq{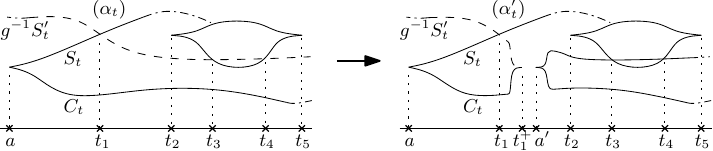}{h!}{1}{Case 1.B - First modification on $ \formes $.}{fig:Case1B}

\underline{Case 2: $t_1 $ is birth.}\label{case:2} The Cerf-Novikov graphic is thus as in the first picture of figure \ref{fig:Case2}, where $\chext{S_t'}{t}{t_1}{t_4}$ represent the chosen liftings of the family of saddles coming from the use of proposition \ref{pro:KSlideToCompet}. Remark that the saddles $s_t$ appearing at $t=a$ are the connecting saddles for the considered family of sinks before $t_2$ and after $t_3$. We start modifying the path by inserting two loops based at times $t_2+\ve $ and $t_3+\ve$ realizing forwards then backwards the elimination of $ s'_{t_2+\ve} $ with $ c_{t_2+\ve} $ and that of $ s_{t_3+\ve} $ with $ c_{t_3+\ve} $ respectively to obtain a situation similar to that of the second picture of figure \ref{fig:Case2} with two swallowtails around $t=t_2,t_3$ respectively. 
\DibLocScalNomEtiq{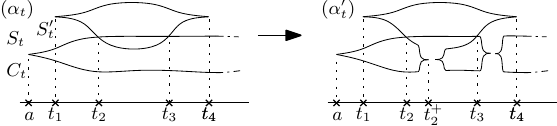}{h!}{1}{Case 2 - First modification on $ \formes $.}{fig:Case2}

As we pointed on proposition \ref{pro:KSlideToCompet}, the separatrices of the inserted saddles $S_t'$ not going to $C_t$ can be supposed to end on a sink $C_t'$ lower than $C_t$, and this for every $t\in [t_1,t_4]$. In particular, the swallowtail concerning the competition at $t=t_3$ can be eliminated as usually. Remark that the situation near the competition at $t=t_2$ does not allow one to apply lemma \ref{lem:SwallowTail} since the sinks $s'_t$ do not exist on the interval $[a,t_1]$.\\

Choose $t_1'<a$. By carefully applying proposition \ref{pro:Shift}, we can shift the index $1$ birth point $S_{t_1}$ to $t=t_1'$ in such a way that the relative position of $\desc{S_t'}$ and $\asc{C_t}$ is preserved for intermediate values of $t$: call $\chext{\cC_t}{t}{t_1-\ve}{t_1+\ve}$ a choice of adapted cylinders of the birth path of index $1$ such that $S'_{t_1}\in\cC_{t_1}$. Remark that $\asc{S_t'}$ is contained on $\adh{\asc{C_t}}\cap\adh{\asc{C'_t}}$  for every $t\in (t_1,t_1+\ve]$ and then, the set
\[
\Delta_t:=\adh{\asc{C_t}}\cup\adh{\asc{C'_t}}
\]
is a connected codimension zero set of $\wt{M}$, also for $t=t_1$. This remains true on the interval $[a,t_1)$, since the open interval does not contain any accident and $\Delta_a$ absorbs $\asc{C_a=S_a}$. Define $\Delta_t:=\wt{M}$ for every $t<a$.\\
Clearly, the set $\Delta:=\cup_{t\in{[0,t_1+\ve]}}\Delta_t$ is a codimension zero connected set of $[0,1]\x\wt{M}$. The $1$-dimensional submanifold $\cT$ appearing on the proof of proposition \ref{pro:Shift} does not disconnect $\Delta$ since $n\geq 3$.\\

Call $a':=t_1'-\ve, b':=t_1-\ve$; up to taking a smaller $\ve$, we can suppose that the interior of the cylinder $\cC_{b'}$ intersects $\Delta_{b'}$. Construct an arc  $\gamma:[a', b']\to\Delta\prive\cT $ as in \eqref{eq:ArcShift} where $L:=(b',x_{b'})$ is taken such that $x_{b'}\in\Delta_{b'}\cap\Int\cC_{b'}$. This path $\gamma$ allows one again to take a continuous path of cylinders parametrized by $[a', b']$ and ending on $\cC_{b'}$, whose union serves as the support of the claimed birth shift preserving the relative position of the mentioned invariant manifolds. After realizing the shift, the effect of the latter modifications is a generic path whose Cerf-Novikov graphic is like the first picture of figure \ref{fig:Case2Modif}.
\DibLocScalNomEtiq{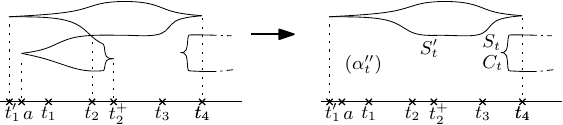}{h!}{1}{Case 2 - Second modification on $ \formes $.}{fig:Case2Modif}

Thanks to the way in which we have performed the birth shift, the hypotheses of lemma \ref{lem:SwallowTail} are verified for some small interval containing $[a,t_2^+]$: we apply it to obtain the claimed path  $ \ch{\al_t''}{t} $. Remark that we have not modified the path near $t_4$; the connecting saddles for the considered family of sinks is thus $s_t$ near $t_4$, as it was at the beginning of this case. Even if  $ \ch{\al_t''}{t} $ is not normal in this case, we can continue the reduction of singular times on the family of sinks $\chext{c_t}{t}{t_4}{b}$ since the elimination at $t_4$ does not create a discontinuity of connecting saddles of the considered sinks.\\

In any case, we find the path $\ch{\al_t^{(1)}}{t}$ with one birth of index $0$ less than the original one. Before treating the next path of sinks, take the precaution of shifting the eventual new index $1$ stabilisation times in order to convert $\ch{\al_t^{(1)}}{t}$ again into a normal path.
\end{proof}
\vspace{1cm}
\textbf{\textsc{Acknowledgements:}} I am indebted by the continuous support I have received from Fran\c{c}ois Laudenbach as well as for sharing with me his valuable point of view of mathematics.\\
I am strongly grateful to Jean-Claude Sikorav for the comments that he made on a previous version of the paper, which have substantially improved the author's understanding about the contributions that have been done to Novikov theory, as well as the accuracy of some citations on the present paper.\\
I greatly appreciate the comments from the anonymous referee, specially for having pointed out an accident that was not treated in the first version of the paper. The first version was partially written at the University of Nantes under the financial support of the \lq\lq ERC GEODYCON project" (eOTP: 12GEODYC). The revisions and improvements which have led to this final version were made during a \lq\lq JSPS Postdoctoral Fellowship" under remarkably good research conditions at The University of Tokyo.

\bibliography{Biblio}{}
\bibliographystyle{alpha}
\vspace{1cm}
\textsc{C. Moraga Ferr\'{a}ndiz:} \newline
Graduate School of Mathematical Sciences, the University of Tokyo\newline
3-8-1 Komaba Meguro-ku, Tokyo 153-8914, Japan\newline
e-mail:  carlos@ms.u-tokyo.ac.jp

\end{document}